\documentclass {article}
\usepackage{graphicx}
\usepackage{amssymb,amsthm,amsmath}
\usepackage{epsfig}
\catcode`@=11 \catcode`!=11

\expandafter\ifx\csname fiverm\endcsname\relax
  \let\fiverm\fivrm
\fi
  
\let\!latexendpicture=\endpicture 
\let\!latexframe=\frame
\let\!latexlinethickness=\linethickness
\let\!latexmultiput=\multiput
\let\!latexput=\put
 
\def\@picture(#1,#2)(#3,#4){%
  \@picht #2\unitlength
  \setbox\@picbox\hbox to #1\unitlength\bgroup 
  \let\endpicture=\!latexendpicture
  \let\frame=\!latexframe
  \let\linethickness=\!latexlinethickness
  \let\multiput=\!latexmultiput
  \let\put=\!latexput
  \hskip -#3\unitlength \lower #4\unitlength \hbox\bgroup}

\catcode`@=12 \catcode`!=12
\input pictex.tex
\catcode`@=11 \catcode`!=11
  
\let\!pictexendpicture=\endpicture 
\let\!pictexframe=\frame
\let\!pictexlinethickness=\linethickness
\let\!pictexmultiput=\multiput
\let\!pictexput=\put

\def\beginpicture{%
  \setbox\!picbox=\hbox\bgroup%
  \let\endpicture=\!pictexendpicture
  \let\frame=\!pictexframe
  \let\linethickness=\!pictexlinethickness
  \let\multiput=\!pictexmultiput
  \let\put=\!pictexput
  \let\input=\@@input   % \@@input is LaTeX's saved version of TeX's primitive
  \!xleft=\maxdimen  
  \!xright=-\maxdimen
  \!ybot=\maxdimen
  \!ytop=-\maxdimen}

\let\frame=\!latexframe

\let\pictexframe=\!pictexframe

\let\linethickness=\!latexlinethickness
\let\pictexlinethickness=\!pictexlinethickness

\let\\=\@normalcr
\catcode`@=12 \catcode`!=12
\def\text#1{\mbox{#1}}

\newtheorem {theo}{Theorem}[section]

\newtheorem {prop}{Proposition}[section]

\newtheorem {cor}{Corollary}[section]
\newenvironment{proofof}[1]{\medskip 
\noindent{\it Proof of #1.}}{ \hfill\qed\\ }

\renewcommand{\rho}{\varrho}
\newcommand\st{\,\,;\,\,}

\newcommand\dist{\mbox{\rm dist}}

\newcommand\rz{{\mathbb R}}
\newcommand\R{{\mathbb R}}
\newcommand\Q{{\mathbb Q}}
\newcommand\zz{{\mathbb Z}}
\newcommand\nz{{\mathbb N}}

\newcommand\vvert{{\vert \vert }}

\begin{document}
\title{Fictitious Play in $3\times 3$ Games: chaos and dithering behaviour}

\author{Sebastian van Strien\footnote{strien@maths.warwick.ac.uk}\,\,  and Colin Sparrow\footnote{c.sparrow@warwick.ac.uk}} 

\maketitle

\begin{abstract}
{In the 60's Shapley provided an example of a
two player fictitious game with  periodic behaviour.
In this game,  player $A$ aims to copy $B$'s behaviour
and player $B$ aims to play one ahead  of player $A$.
In this paper we continue to study a family of 
games which generalize Shapley's example by introducing
an external parameter, and prove that there exists an
abundance of periodic and chaotic behavior with players
dithering between different strategies. The reason for all this,
is that there exists a periodic orbit (consisting of playing mixed strategies)
which is of {\em `jitter type'}: such an orbit is neither attracting, repelling or of saddle type
as nearby orbits jitter closer and further away from it in a manner which is
reminiscent of a random walk motion. 
We prove that this behaviour holds for an open set of games.} 
\end{abstract}

%\tableofcontents

\section{Introduction}

The purpose of this paper is to show how complicated the dynamics of 
fictitious play can be (for an interpretation of fictitious play as a model for rational learning,
see for example Fudenberg and Levine \cite{Fudenberg-Levine98}). We do this by 
analysing in detail the following family of  $3\times 3$ games determined by the matrices 
\begin{equation}
A_\beta= \left(\begin{array}{ccc} 1 & 0 & \beta \\ \beta & 1 & 0 \\ 0 & \beta & 1
\end{array} \right) \quad
B_\beta=\left(\begin{array}{ccc}  -\beta & 1 & 0 \\ 0 & -\beta & 1 \\ 1 & 0 & -\beta
\end{array} \right),
\label{eqn:AB}
\end{equation}
which depend
on a parameter $\beta\in (0,1)$ (and best response 
dynamics given by the differential inclusion (\ref{eq1})).
In fact,  we shall show in Theorem~\ref{thm:robustness}
that our results even hold for  matrices $A,B$ with 
$$||A-A_\beta||,||B-B_\beta|| \le \epsilon \mbox{ with }\epsilon>0 \mbox{ small}.$$
However, except for Theorem~\ref{thm:robustness} and Section~\ref{sec:robustness}, we shall 
simply write $A=A_\beta$ and $B=B_\beta$.
As usual, player $A$ has utility $p^A A p^B$ whereas player
$B$ has utility $p^A B p^B$ where %%%:=p^AB$ where 
the row vector 
$p^A\in \Sigma_A\subset \rz^3$ denoted the position of player $A$ and the column
vector $p^B\in \Sigma_B\subset \rz^3$ the position of player $B$.
For later use we write $v^A=Ap^B$ and $v^B=p^A B$.
Here $\Sigma_A,\Sigma_B$ are the set of probability vectors in $\rz^3$.
In other words, $\Sigma:=\Sigma_A\times \Sigma_B$ is the product of two-dimensional 
triangles and so topologically it is a ball in $\rz^4$. Player $B$ (resp. $A$) is indifferent 
between all three strategies when $p^A=E^A$, $p^B=E^B$ 
and $E=(E^A,E^B)$ is the Nash equilibrium of the game. 
For the game $A=A_\beta$, $B=B_\beta$ one has
$E^A:=(1/3,1/3,1/3)$ and $E^B:=(1/3,1/3,1/3)'$.

For $\beta=0$ the game $A,B$ is equivalent to the classical example introduced by Shapley
\cite{Shapley64} (where each of the players eventually chooses strategies 
periodically). 
For $\beta=(\sqrt{5}-1)/2\approx 0.618$, the game is equivalent to a zero-sume game
(rescaling $B$ to $\tilde B=\sigma(B-1)$ gives $A+\tilde B=0$), so then \cite{Robinson51}
play always converges to the interior equilibrium $E^A,E^B$.
 
The best response $BR_A(p^B)$ of player $A$ is the $i$-th unit vector
if the $i$-th component of $v^A$ is larger than the other components of $v^A$
(if several components of $v^A$ are equally large, then $BR_A(p^B)$
is the convex combination of unit vectors corresponding to %collection of 
the largest components of $v^A$). Define $BR_B(p^A)$ similarly.
Once best responses are selected, the dynamics is determined by moving
in a straight line towards
the best responses. In some of the literature this is done by taking the 
piecewise linear differential equation
\begin{equation}
\begin{array}{ll}
       {dp^A}/{dt} & =  BR_A(p^B) - p^A \\
       {dp^B}/{dt} & =  BR_B(p^A) - p^B \end{array}
\label{eq1}\end{equation}
whereas others take 
\begin{equation}
\begin{array}{ll}
       {dp^A}/{ds} & = \,\,  (1/s)  \, (BR_A(p^B) - p^A) \\
       {dp^B}/{ds} & = \,\, (1/s)  \, (BR_B(p^A) - p^B) . \end{array}
\label{eq2}\end{equation}
The orbits are the same in both cases, only the time parametrisation of the orbits
differs (take $s=e^t$);  for the latter, orbits slow down and a periodic orbit of period $T$ 
(\ref{eq1})  corresponds to an orbit of (\ref{eq2}) which returns in time 
$e^T,e^{2T},e^{3T},\dots$. 
Equations (\ref{eq1}) and (\ref{eq2}) determine the dynamics up {\em until such time}
as one or other (or both) players become {\em indifferent} 
between two (or more) pure strategies. When one or more of the players is 
indifferent between two strategies their dynamics may not be uniquely determined.
So (\ref{eq1}) and (\ref{eq2}) are in fact differential inclusions rather than differential equations,
but as the best response correspondences 
$p^B\mapsto BR_A(p^B)$ and $p^A\mapsto BR_B(p^A)$
are upper semicontinous with values closed, convex sets, 
it follows from Aubin and Cellina \cite[Chapter 2.1]{AubinCellina84}
that through each initial value there exists at least one solution
which is Lipschitz continuous and defined for all positive 
time.  It is shown in Hofbauer \cite{Hofbauer95}  that, under mild regularity 
conditions (which are satisfied in our case), any solution is piecewise linear.

In fact,  when the matrices (\ref{eqn:AB}) are chosen, play is not
affected at all by this ambiguity except at $E$ 
(because a certain transversality condition
is satisfied, see  Sparrow \cite{SSH2008}).  In other words, all the orbits  
(except the one through $E$) are  uniquely determined (outside $E$, the dynamics is not affected
by a choice of tie-breaking rule). Moreover,  when $\beta\in (0,1)$
the flow is continuous except at $E$ (when $\beta\ne \sigma$, 
for a proof,  see Sparrow et al \cite{SSH2008}).

Note that the best response of $A$ to any $p^B\ne E^B$ is either an integer
$i\in \{1,2,3\}$ or a mixed strategy set $\bar i$ where $\bar i:=\{1,2,3\}\setminus \{i\}$
corresponding to where player $A$ is indifferent between two strategies but will {\bf not} play $i$.
Similarly for $B$. Hence one can associate to any orbit $(p^A(t),p^B(t)$ outside $E$, 
a sequence of times $t_0:=0<t_1<t_2<\dots$ and a sequence of best-response strategies 
$(i_0,j_0),(i_1,j_i),(i_2,j_2),\dots$ where 
$$(i_n,j_n)=(BR^A(p^B(t),BR^B(p^A(t))\mbox{ for }t\in (t_n,t_{n+1})$$
with  $i_n$ and $j_n$ equal to $1$, $2$, $3$, $\bar 1$, $\bar 2$ or $\bar 3$ for each $n=0,1,2,\dots$.
In Sparrow et al \cite{SSH2008} we showed that there are three periodic 
orbits: one for $\beta\in (0,\sigma)$ with play $(1,2),(2,2)$, $(2,3)$, $(3,3)$, $(3,1)$, $(1,1)$
(the Shapley orbit), one for $\beta\in (\sigma,1)$ 
with cyclic play $(1,3),(1,2), (3,2),(3,1)$, $(2,1),(2,3)$
and a third one with a period 6 orbit of mixed strategies 
$(\bar 1, \bar 1), (\bar 1, \bar 2), (\bar 2, \bar 2), (\bar 2, \bar 3), (\bar 3, \bar 3), (\bar 3,\bar 1)$.
The latter sequence of strategies correspond to a   fully-invariant set $C(\Gamma)$ 
(so an orbit starting in this set  remains in this set, and an orbit starting outside this 
set remains outside this set); this fully invariant set exists for each $\beta\in (0,1)$
and   contains a periodic orbit when $\beta\in (\sigma,1)$.
The latter orbit is of {\em `jitter type'}: it is neither attracting, repelling or of saddle type;
instead  nearby orbits jitter closer and further away from it in a manner which is
reminiscent of a random walk motion. We describe this behavior in the final sections
of this paper.

\subsection{Abundance of Periodic Play}
Let us state now the main results of this paper.
To do this, let us say that an orbit of the game has {\em cyclic play of period $n$} if the associated
sequence $(i_n,j_n)$ is periodic: $(i_{k+n},j_{k+n})=(i_k,i_k)$ for all $k=0,1,\dots$.
Given  $k\in \{0,\dots,n-1\}$ we say that the players are {\em indecisive at the $k$-th step} 
if $k\ge 3$ and moreover 
$$\{i_{k-3},i_{k-2},j_{k-1},i_{k}\}\ne \{1,2,3\}
\mbox{ {\bf and } }\{j_{k-3},j_{k-2},j_{k-1},j_{k}\}\ne \{1,2,3\}$$
holds (so during this and the previous three moves, both players never deviated from a choice of two strategies). 
Sometimes we also will say that the players {\em dither} at the $k$-step.
In the opposite case, we say the $k$-th step is {\em decisive}.
% if $k\le 2$ or the players are not indecisive at step $k$.
The {\em essential period} of a  cyclic play of period $n$ is 
the number of decisive steps $k\le n$.
For example,  the cycle of period $6$
$$(1,2),(2,2), (2,3),(3,3),(3,1),(1,1)$$
never dithers whereas for the cycle of period $7$
$$(1,2),(2,2), (2,3),(3,3),(3,2),(3,1),(1,1)$$
players dither in the 5th step, so the essential period is again 6. 
Let $\sigma=(\sqrt{5}-1)/2\approx 0.618$.

\begin{theo}\label{thm:main} [An abundance of periodic play]
For each $\beta\in (0,1)$ and {\em each} $n\ge 1$
there are infinitely many different orbits $\gamma_s$, $s=1,2,\dots$ 
of the differential equation
(\ref{eq1}) (and of (\ref{eq2})) with corresponding cyclic play of period $N_s\to \infty$ as $s\to \infty$ but with essential period equal to $6n$. Moreover,
\begin{itemize}
\item for $\beta\in (0,\sigma)$, these orbits with cyclic play reach the interior equilibrium $E$ in finite time;
\item
for $\beta\in (\sigma,1)$ these orbits  with cyclic play are genuine periodic orbits of
(\ref{eq1})  (and of (\ref{eq2})).
\end{itemize}
\end{theo}

In Sparrow et al \cite{SSH2008} we showed that for $\beta\in (-1,\sigma)$ there exists a periodic 
orbit corresponding to cyclic play $(1,2),(2,2), (2,3),(3,3),(3,1),(1,1)$
(the Shapley orbit)
which attracts an open set of initial conditions; the  above theorem 
shows  that many periodic orbits are not attracted to this cycle.
In that paper it was also shown that for $\beta\in (\sigma,1)$ there exists another periodic orbit
corresponding to cyclic play $(1,3),(1,2), (3,2),(3,1),(2,1),(2,3)$
(the anti-Shapley orbit) which becomes attracting when $\sigma\in (\tau,1)$ where 
$\tau\approx 0.915$. Again this attracting orbit does not attract everything.

That the players can have infinitely many orbits with the same essential period,
is a consequence of the fact that there is a sequence of periodic orbits 
converging to the   %%period 6 
orbit of mixed strategies
$(\bar 1, \bar 1), (\bar 1, \bar 2), (\bar 2, \bar 2), (\bar 2, \bar 3), (\bar 3, \bar 3), (\bar 3,\bar 1)$,
and which all have the same essential period. For these periodic orbits, the essential period is 
the number of times it follows the period 6 orbit before returning to its original position,
whereas the actual period increases if the players dither for longer along each of the 
6 legs. Along these periodic orbits at any given moment only one of the players is indifferent,
but they dither for a long time between each decisive move. 

Let us relate this theorem to a result of Krishna and 
Sj\"ostr\"om \cite{Krishna-Sjostrom98} (which  
builds on earlier work of Rosenm\"uller \cite{Rosenmuller71}).
In this interesting paper, they show that for a  generic game
(i.e. for Lebesgue almost all pay-off matrices) there exists
no open set of initial conditions for which
fictitious play converges cyclically to a mixed strategy
equilibrium  (unless  both players use at most two pure strategies).
In other words, if fictitious play converges to a mixed strategy
equilibrium with both players using more than two strategies 
then the choice of strategies cannot follow a cyclic  pattern
unless possibly the initial conditions are in some codimension-one space.
Our result shows that for $\beta\in (0,\sigma)$ countably many orbits 
do indeed converge cyclically to the equilibrium (and along these orbits 
at any given moment only one of the players is indifferent).
Our result does not rely on the symmetry of the matrices:
it holds for an open set of matrices,  see Theorem~\ref{thm:robustness} and Section~\ref{sec:robustness}.

For their result, Krishna and Sj\"ostr\"om {\em only} only   needed to consider orbits for which
 at any given moment only one of the players is  indifferent. 
 It is interesting to note however that, as becomes clear from 
 this paper, it is precisely near the set $C(\Gamma)$ 
 where {\em both} players are simultaneously indifferent that much of the interesting
behaviour happens (and this set 'organises' the local dynamics).

\subsection{Abundance of Dithering Behavior}

The next theorem shows that there are many orbits which
dither for very long periods. To make this precise,
let us assume the players start at $p\in \Sigma$ and
aim for $(i_0,j_0),(i_1,j_1),(i_2,j_2),\dots$.
Next associate to these moves a sequence $(R_k)_{k\ge 3}$
with $R_k\in \{D,I\}$, where $R_k$ is equal to $D$ or $I$
depending on whether the players are decisive  or indecisive at time $k\ge 3$.
In this way we get a map 
$$\Sigma\ni p \mapsto \{R_k(p)\}_{k\ge 3}\in \{D,I\}^\nz$$
which captures partly what play evolves from starting position $p\in \Sigma$.
(We ignore $k<3$ because by definition the players are then
always decisive. More precisely, % We take $k\ge 3$ for the following reason: 
if $T$ denotes the map which assigns to $(i_k,j_k)_{k\ge 0}$ the sequence  $(R_k)_{k \ge k_0}$
and $\sigma$ the shift map, then $T\circ \sigma=\sigma \circ T$ only if
 we take $k_0\ge 3$.)

To simplify the coding even further, 
define the times $3\le N_0(p)<N_1(p)<N_2(p)<\dots$
for which the players are decisive (only considering times $\ge 3$).
Note that these times uniquely determine again the sequence $R_3(p),R_4(p),\dots$\, .
If $N_{s+1}(p)-N_s(p)$ is large, then we say that the players
{\em dither} for a long time (as they then 
each play back and forth between two strategies).

\begin{theo}
\label{thm:dither}
[There is a lot of freedom in the choice of dithering sequences]
% [Players can dither almost at will]
For each $\beta\in (0,1)$ there exist $N^0\in \nz$, $0<\lambda<1<\mu$ and  
a compact set $X\subset 
\Sigma$, 
%so that for each $p\in X$
%$$N_{2s+2}(p)-N_{2s}(p)\ge N^0 \mbox{ for all }s\ge 0$$
%and 
so that for each  sequence $N_0<N_2<N_4<\dots$
with $N_0\ge 3$ and with
$$\lambda \le \frac{N_{2s+4}-N_{2s+2}}{N_{2s+2}-N_{2s}}\le \mu \mbox{  and }
N_{2s+2}-N_{2s}\ge N^0 \mbox{ for all }s\ge 0$$
there exists  $p\in X$ such that
$$|N_{2s}(p)-N_{2s}|\le 4 \mbox{ for all }s\ge 0.$$
Moreover,
\begin{itemize}
\item
for $\beta\in (0,\sigma)$  orbits in $X$ converge to $E$;
\item
for $\beta\in (\sigma,1)$  orbits starting in $X$ do {\em not} converge to $E$.
\end{itemize}
\end{theo}

\medskip

That this theorem only refers to the gaps between even decisive moments, $N_{2s+2}-N_{2s}$,
is because the gaps between the even and odd moments are somewhat more arbitrary.
However, if $N_{2s+2}(p)-N_{2s}(p)$ is large for all $s\ge 0$, then 
$N_{s+1}(p)-N_s(p)$ is also large for all $s\ge 0$, i.e., 
the players dither for long periods between making a decisive move for orbits described above.

We did some numerical simulations for games determined by completely different
matrices.  In many of these, similar dithering behaviour also occurred.
Even for zero-sum games, the players seem to converge to equilibria in a dithering fashion
(in fact, in $2\times n$ games, dithering is unavoidable).
We will report on these simulations in a subsequent paper.

\subsection{Chaotic Behavior}

The previous theorem states that there are orbits starting in $X$ which dither for more or less
arbitrary lengths $N_{2s+2}-N_s$.  An immediate application of this theorem
is the following result:

\begin{theo}[Chaos]
Take  
$\hat N\ge N^0$ so large that $\lambda\le (\hat N-1)/\hat N \le (\hat N+1)/\hat N\le \mu$.
For each sequence $(\epsilon_{2s})_{s\ge 0}$ with $\epsilon_{2s}\in \{-5,0,5\}$
there exists $p\in X$ with
$$N_{2s+2}-N_{2s}\in [\hat N + \epsilon_{2s},\hat N+4 +\epsilon_{2s}] \mbox{ for all }s\ge 0.$$
So for such a $p$, \,\, $N_{2s+2}-N_{2s}$ is in the interval $[\hat N-5,\hat N-1]$, $[\hat N,\hat N+4]$
or $[\hat N+5,\hat N+9]$ depending on the parity of $\epsilon_{2s}$. 
In particular the flow contains subshifts of finite type and has positive topological entropy.
The flow also has sensitive dependence on initial conditions.
\end{theo}

The definition of the notions `subshifts of finite type',   `positive topological entropy'
and `sensitive dependence on initial conditions' can be found in almost any book
on dynamical systems, for example Guckenheimer and Holmes \cite{GH}.
We have numerical evidence that
this game is  chaotic in a more profound sense: it appears that there exists a range 
of parameters $\beta\in (\sigma,\tau)$ so that (Lebesgue) almost all starting positions
correspond to chaotic behaviour. We will report on this in a subsequent paper.

\subsection{Robustness}

The above results do not require the matrices to be of a
special form, and hold for games corresponding to an open set of matrices:

\begin{theo}[Robustness]
\label{thm:robustness}
For each $\beta\in (0,1)$ with $\beta\ne \sigma$,
there exists $\epsilon>0$ so that
for each $3\times 3$ matrices $A$ and $B$
with $$||A-A_\beta ||,||B-B_\beta ||<\epsilon$$
the previous theorems also hold. More precisely,
\begin{itemize}
\item for $\beta\in (0,\sigma)$ and $\epsilon>0$ sufficiently small,  
\begin{itemize}
\item there exists a periodic  orbit corresponding to cyclic play $(1,2)$, $(2,2)$, $(2,3)$, $(3,3)$, $(3,1)$, $(1,1)$
(the Shapley orbit) which attracts an open set of initial conditions; 
\item there exist orbits of mixed strategies 
$(\bar 1, \bar 1), (\bar 1, \bar 2), (\bar 2, \bar 2), (\bar 2, \bar 3), (\bar 3, \bar 3), (\bar 3,\bar 1)$;
such orbits lie on a cone with apex $E$; all orbits on this cone converge to $E$;
there are infinitely many different orbits $\gamma_s$, $s=1,2,\dots$  as in Theorem~\ref{thm:main}
which reach the interior equilibrium $E$ in finite time;
\item there are orbits which dither as in Theorem~\ref{thm:dither} which again reach the interior  equilibrium $E$ in finite time;
\end{itemize}
\item  for $\beta\in (\sigma,1)$ and $\epsilon>0$ sufficiently small,  
\begin{itemize}
\item there exist a periodic orbit $\Gamma$ of mixed strategies 
$(\bar 1, \bar 1), (\bar 1, \bar 2), (\bar 2, \bar 2), (\bar 2, \bar 3), (\bar 3, \bar 3), (\bar 3,\bar 1)$;
the cone through this periodic orbit with apex $E$ is completely invariant and all orbits on this cone (apart 
from $E$) converge to this periodic orbit;
\item there are infinitely many different periodic orbits $\gamma_s$, $s=1,2,\dots$  as in Theorem~\ref{thm:main}
(these orbits stay near $\Gamma$);
\item  there are orbits which dither as in Theorem~\ref{thm:dither} (these orbits also  stay near $\Gamma$);
\item one has chaos and sensitive dependence on initial conditions;
\end{itemize}
\item  for $\beta\in (\tau,1)$ and  $\epsilon>0$ sufficiently small,  
there exists an attracting orbit which corresponding to cyclic play $(1,3),(1,2), (3,2),(3,1),(2,1),(2,3)$
(the anti-Shapley orbit). 
\end{itemize}
\end{theo}

\bigskip

Here $\tau\approx 0.915$ is the root of some
polynomial of degree 6, which we computed in Sparrow  et al \cite{SSH2008}.
Of course $A$ and $B$ near $A_\beta$ resp. $B_\beta$ 
will have a Nash equilibrium $E$, which is close but not necessarily
equal to the Nash equilibrium  of $A_\beta$,  $B_\beta$. 
For matrices $A,B$ near $A_\sigma,B_\sigma$ the existence
of periodic orbits and of a dithering set also hold,
but it is no longer clear whether these orbits converge to the Nash equilibria $E$ or not.

\subsection{The idea of the proof and some general comments}

The main point of our analysis is to exploit that one can simplify the study
by identifying points on half-lines through $E$. This way we get an induced
flow on $\partial \Sigma$ (which is topologically a three sphere).
Associated to each set 
$Y\subset \partial \Sigma$ which is forward invariant under the induced flow
is the cone $C(Y)$ over $Y$ with apex $E$
which is forward invariant under the original flow.
Similarly,  a periodic orbit $\gamma$ of the original flow, 
corresponds to a periodic orbit $\tilde \gamma$ of the induced flow.
We apply this idea in particular to the periodic
orbit $\Gamma$  with mixed strategies
$(\bar 1, \bar 1), (\bar 1, \bar 2), (\bar 2, \bar 2), (\bar 2, \bar 3), (\bar 3, \bar 3), (\bar 3,\bar 1)$
and the  corresponding periodic orbit $\tilde \Gamma$ of the induced flow. 
It turns out that a first return map to a section through a point
in $\tilde \Gamma$ has extremely interesting behaviour:
it is of {\em `jitter type'}, see the final section
of this paper.

We believe that looking at our approach of analysing 
the induced flow, and the notion of orbits of `jitter type'  will be useful in analysing fictitious play in general.

This is not the first time subshifts of finite type were shown to exist 
in fictitious play.  Cowan \cite{Cowan92} already did this,
by considering a matrix with extremely large coefficients. 
Our work is closer in spirit to Berger \cite{Berger1995} who
considers a  family of symmetric bimatrix games
depending on a parameter $k$ such that
(i) for $k\in (-2,0)$ has a Shapley orbit  which
degenerates as $k\to 0$, and (ii) for $k>0$, as in our case, there exists 
a hexagonal orbit along which both players are indifferent between (at least) 2 strategies. Berger observes similar 'chaotic' numerical phenomena as we did in our
previous paper Sparrow et al \cite{SSH2008}  
and also shows the existence of an additional periodic orbit of saddle-type.
  
As mentioned, we have numerical evidence that  there exists a range 
of parameters $\beta\in (\sigma,\tau)$ so that (Lebesgue) almost all starting positions
correspond to chaotic behaviour. More general games also show up 
the same dithering behaviour.

There are many papers which show that one has convergence to the equilibrium 
for  games where one or both of the players have only 2 strategies to choose from,
see Miyasawa \cite{Miyasawa61} and Metrick \& Polak 
\cite{Metrick-Polak94}
for the $2\times 2$ case; Sela \cite{Sela00} for the $2\times 3$ case; and
Berger \cite{Berger2005} for the general $2\times n$  case.
Jordan \cite{Jordan93} constructed a $2\times 2 \times 2$ fictitious game with a stable limit cycle.  
The $3\times 3$ example studied in this paper,  shows that the situation is far more complicated in 
general.

%relationship between the replicator dynamics and fictitious play 
%In the GENERIC case as discussed above, player A continuously
%adjusts her probability vector $p^A(t)$ from its current position $p^A(t_0)$ 
%in a straight line towards $P^A_i$, and player B similarly adjusts $p^B(t)$ to
%move in a straight line towards $P^B_j$, so that for times greater than 
%$t_0$, and for so long as strategies $P^A_i$ and $P^B_j$ 
%remain unique best responses,  
%we may write, after a reparameterisation of time, the solution of 
%(\ref{eq1}) becomes
%\begin{equation}
%\begin{array}{ll}
%p^A(t_0+s)&=p^A(t_0)(1-s) + s \cdot P^A_i,\\
%p^B(t_0+s)&=p^B(t_0)(1-s) +  s \cdot P^B_j,
%\end{array}
%\label{eq2}
%\end{equation}
%for $s \in [0,1]$.  
%Note that we have parameterised time 
%so that both players move towards their best response strategies
%at a uniform speed, such that if their best responses did
%not change they would arrive at $(P^A_i, P^B_j)\in \partial \Sigma$ at time $t=t_0+1$, 
%at which point
%the dynamics would halt.  The parameterisation is
%chosen for ease of calculation and computation, and does not affect the
%geometry of the dynamics which is our chief concern in this paper; it would
%be straightforward to use the original time $\rho = -\ln (1-s)$ if
%an exponential approach to $(P^A_i, P^B_j)$ in infinite time were preferred
%(the speed of a solution of (\ref{eq1}) is proportional to the distance to the target point).  
%In any case, since orbits consist of pieces of line segments, numerical simulations
%can be done easily and with arbitrary accuracy.

\section{Basic results on the Shapley system}
%\label{s:simpleprops}

%
%From now we shall concentrate on  the
%particular family of examples (\ref{eqn:AB}) defined above, which includes a classical
%example by Shapley.  The purpose of this paper is to
%show that these games have extremely complicated dynamics.
%To be able to state these results, we will first review results from our previous paper
%Sparrow et al \cite{SSH2008}.

%
%\begin{equation}
%A= \left(\begin{array}{ccc} 1 & 0 & \beta \\ \beta & 1 & 0 \\ 0 & \beta & 1
%\end{array} \right) \quad
%B=\left(\begin{array}{ccc}  -\beta & 1 & 0 \\ 0 & -\beta & 1 \\ 1 & 0 & -\beta
%\end{array} \right),
%\label{eqn:AB}
%\end{equation}

Let us recall some results from Sparrow et all \cite{SSH2008}.
Let us denote the set where player $A$ is indifferent between strategies $P_i^A$ and $P_j^A$ by $Z_{i,j}^A\subset \Sigma_B$ and define $Z_{ij}^B$ similarly. 
Figure~\ref{fig:Jorbit} shows pictures of the phase space marking
the lines $Z_{ij}^A$ and $Z_{ij}^B$ (for $\beta>0$).
%(In Figure~\ref{fig:ssb=0} the phase space was drawn for $\beta=0$.)
Note that for all values of $\beta$ both players are 
indifferent between all three strategies
at the point $E=(E^A, E^B)$ where  $E^A=(E^B)^T=(1/3,1/3,1/3)$.
%Outside the set $Z=(\Sigma_A\times \cup_{ij}Z_{ij}^A)\times (\cup_{ij}Z_{ij}^B \times \Sigma_B)$
%the flow defined by (\ref{eq2}) is uniquely defined and continuous. 

As mentioned, for $\beta\in [0,\sigma)$ the game has a periodic orbit
with cyclic play  $(1,2)$, $(2,2)$, $(2,3)$, $(3,3)$, $(3,1)$, $(1,1)$ and this orbit attracts an 
open set. For $\beta\in (\sigma,1)$ the game has another periodic orbit 
with cyclic play $(1,3),(1,2), (3,2),(3,1),(2,1),(2,3)$ and this orbit
is attracting for $\beta\in (\tau,1)$. When $\beta\to \sigma$
these orbits shrink to $E$. A third periodic orbit  $\Gamma$ exists
when $\beta\in (\sigma,1)$ with periodic 6 cyclic play 
$(\bar 1, \bar 1)$, $(\bar 1, \bar 2)$, $(\bar 2, \bar 2)$, $(\bar 2, \bar 3)$, $(\bar 3, \bar 3)$, $(\bar 3,\bar 1)$.
For $\beta\in (0,\sigma)$ there still exist orbits with this periodic 6 play,
but these orbit are not periodic, instead they 
converge to $E$. Two of the six sides of the corresponding
hexagon are schematically drawn in Figure~\ref{fig:Jorbit}.
Note that $\Gamma$ is contained in
$$J:=
(Z^B_{1,2}\times Z^A_{3,1})\cup (Z^B_{1,2}\times Z^A_{1,2})
\bigcup
(Z^B_{2,3}\times Z^A_{1,2})\cup (Z^B_{1,2}\times Z^A_{2,3})\bigcup
$$
\begin{equation}
\bigcup
(Z^B_{3,1}\times Z^A_{2,3})\cup (Z^B_{2,3}\times Z^A_{3,1})
\label{eqn:defJ}
\end{equation}
We call this the Jitter set, as nearby orbits jitter back and forth between strategies.
 When $\beta\to \sigma$ these orbits all shrink to $E$.

\begin{figure}[htp]
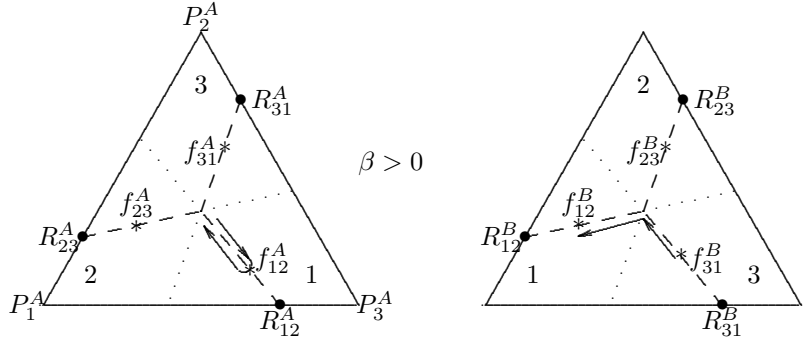
 \hfil
\beginpicture
\dimen0=0.21cm
\setcoordinatesystem units <\dimen0,\dimen0>
\setplotarea x from 0 to 30, y from -5 to 20
\put {$\beta>0$} at 22 9
\setlinear
\setsolid
\plot 0 0 20 0 /
\plot 0 0 10 17.3 20 0 /
\setdashes
\plot 15 0 10 5.98 /
\plot 2.5 4.33 10 5.98 /
\plot 12.5 13 10 5.98 /
\setdots
\plot 10 5.98 6.14 10.60 /
\plot 10 5.98 15.80 7.26 /
\plot 10 5.98 7.88 0 /
\setsolid
\multiput {*} at 11.5 9.8 /  %11.4 9.8 /
\multiput {*} at 13.1 2.0 5.9 4.8 /
\multiput {$\bullet$} at 15 0 2.5 4.33 12.5 13 /
\put {$f_{12}^A$} at 14.5 3
\put {$f_{31}^A$} at 10 9.8
\put {$f_{23}^A$} at 5.9 6.5
%\put {$\alpha^A$} at 13 1       % added
%\put {$\alpha^A+\pi/3$} at 17.4 1 % added
%\put {$\frac{2\pi}{3}-\alpha^A$} at 5 9 % added
\put {$R_{12}^A$} at 15 -1
\put {$R_{31}^A$} at 14.5 13
\put {$R_{23}^A$} at 1 4.33
%\put {$Q^A_{12}$} at 5 10.60
%\put {$Q^A_{23}$} at 16.80 7.26
%\put {$Q^A_{13}$} at 7.88 -1
\put {$P^A_1$} at -1 0
\put {$P^A_3$} at 21 0
\put {$P^A_2$} at 10 18.3
\put {$2$} at  3 2
\put {$3$} at 10 14
\put {$1$} at 17 2
\arrow <5pt>  [0.2,0.4] from 11 5.5 to 13 3
\arrow <5pt>  [0.2,0.4] from 12.3 2.3 to 10.2 5.0
\circulararc -180 degrees from 13 3 center at 12.7  2.6
\setcoordinatesystem units <\dimen0,\dimen0> point at -28 0
\setplotarea x from 0 to 30, y from -5 to 20
\setsolid
\plot 0 0 20 0 /
\plot 0 0 10 17.3 20 0 /
\setdashes
\plot 15 0 10 5.98 /
\plot 2.5 4.33 10 5.98 /
\plot 12.5 13 10 5.98 /
\setdots
\plot 10 5.98 6.14 10.60 /
\plot 10 5.98 15.80 7.26 /
\plot 10 5.98 7.88 0 /
\setsolid
\multiput {*} at 11.4 9.8 12.4 3 5.9 4.9 /
\multiput {$\bullet$} at 15 0 2.5 4.33 12.5 13 /
\put {$f_{31}^B$} at 14.0 3
\put {$f_{23}^B$} at 10 9.8
\put {$f_{12}^B$} at 5.9 6.5
\put {$R_{31}^B$} at 15 -1
%\put {$\alpha^B$} at 13 1       % added
%\put {$\alpha^B+\pi/3$} at 17.4 1 % added
%\put {$\frac{2\pi}{3}-\alpha^B$} at 5 9 % added
\put {$R_{23}^B$} at 14.5 13
\put {$R_{12}^B$} at 1 4.33
%\put {$Q^B_{13}$} at 5 10.60
%\put {$Q^B_{12}$} at 16.80 7.26 %%% WAS WRONG SSSSS
%\put {$Q^B_{23}$} at 7.88 -1 %%% WAS WRONG SSSSS
%\put {$P^B_1$} at -1 0
%\put {$P^B_3$} at 21 0
%\put {$P^B_2$} at 10 18.3
\put {$1$} at  3 2
\put {$2$} at 10 14
\put {$3$} at 17 2
\arrow <5pt>  [0.2,0.4] from 12.0 3 to  10 5.5
\arrow <5pt>  [0.2,0.4] from 10 5.5 to  5.9 4.4
\endpicture
\caption{\label{fig:Jorbit}
{\small The simplices $\Sigma_A$ and $\Sigma_B$ and the periodic orbit on the 
jitter set $J$ (for $\beta\in (\sigma,1)$).  
%The regions where player $A$ prefers the $i$-th strategy are the part of the triangle $\Sigma_B$
%marked by $i$ (and whose boundary are the dashed-curves), and similarly for player $B$
%they are drawn in $\Sigma_A$. 
%Various points on the  indifference sets $Z^A$ and $Z^B$ (and their extensions) are marked.
%$R^A_{23}=\frac{1}{2+\beta} (1+\beta,1,0),$
%$Q^A_{12}=\frac{1}{1+2\beta} (\beta,1+\beta,0)$
%(and similarly formulas for the other points
%by cyclically permuting coordinates).
%Similarly
%$R^B_{12}=\frac{1}{2-\beta}(1,1-\beta,0)$
%and
%$Q^B_{13}=\frac{1}{1+\beta}(\beta,1,0).$
Two of the six legs of the periodic orbit are drawn in the figure (contained
in $Z_{1,2}^B\times Z_{1,3}^A$ and $Z_{1,2}^B\times Z_{1,2}^A$);
after this piece of the orbit,  the orbit turns in $\Sigma_A$ 
clockwise to the leg containing $f^A_{23}$ and in $\Sigma_B$ it doubles back
onto itself.  The points $f^A_{ij}$, $f^B_{kl}$ give the maximum extent
of the periodic orbit on $J$ and $R_{ij}^A$ and $R_{ij}^B$ are the intersections
of $J$ with the boundary of $\Sigma_A$ resp. $\Sigma_B$.
In this way
the periodic orbit  forms a hexagon in the 
four-dimensional space $\Sigma_A\times \Sigma_B$ which
moves (cyclically) around the different legs of $J$. 
}}
\end{figure}

\section{The induced dynamics on $\partial \Sigma$ for the Shapley family}
\label{s:induced}
 
In order to get a better understanding of the geometric and topological 
structure of all orbits in the Shapley family we will now consider an induced
flow on $\partial \Sigma$. This is obtained by projecting the original flow on $\Sigma
\setminus \left\{ E\right\} $ onto $\partial \Sigma $ using the
projection $\pi\colon \Sigma\setminus \{E\} \to \partial \Sigma $ obtained by defining 
$\pi \left( p\right) $ to be the
unique point of $\partial \Sigma $ that lies on the half line through $E$ in
the direction $p$.  Since the best response of the two players for all points on this half-line 
is the same, this gives a well-defined flow on $\partial \Sigma$.
The new flow obtained in this way is a faithful
representation of the 'angular' component of the original flow, but it
contains no information about the radial component of the flow. It is easier
to visualize because it is three dimensional. For example, the topological
three sphere $\partial \Sigma $ is homeomorphic (via stereographic projection) to the one-point
compactification of ${\rz}^{3}$.

The geometry of this induced flow will give
us more insight in the original flow. Indeed let
 $\gamma\subset \partial \Sigma$ and let $C(\gamma)\subset \Sigma$ be 
union of the closed half-lines through $E$ in the direction of $p\in \gamma$ (for any such $p$).
This set we will call the {\em cone of $\gamma$ over $E$} (this set is equal to the closure of
$\pi^{-1}(\gamma)$). Note that if $\gamma$ is a periodic orbit of the induced flow, then
the cone $C(\gamma)$ is an invariant set under the original flow.

%------------------SUBSECTION----------------
\subsection{Simple periodic orbits on the induced flow on $\partial \Sigma$}

Even though the flow on $\Sigma_A\times \Sigma_B$
does not have a periodic orbit on the jitter-set 
$J$ for $\beta\in (0,\sigma)$,
the induced flow does have a periodic orbit on $J\cap \partial \Sigma$. The reason
for this is that $J$ is invariant, so the closed curve $J\cap \partial 
\Sigma$ is an 
invariant set for the induced flow.
When $\beta\in (0,\sigma)$ then orbits in $J$ of the original flow spiral 
towards $(E^A,E^B)$ (inside the topological surface $J$;
but the motion is never straight towards $(E^A,E^B)$).
Therefore in the induced flow this spiralling motion corresponds
to a periodic motion on the closed curve $\partial \Sigma \cap J$.
Let us denote this periodic orbit by $\tilde \Gamma$.

We can summarize the previous results on the existence and stability of periodic orbits
as follows.

\begin{prop} [Periodic orbits for the 
induced flow on $\partial \Sigma$]
\label{prop:orbonS3}
For $\beta\in (0,1]$ there are (at least)
three periodic orbits for the induced flow:
the one corresponding to the clockwise periodic orbit
(Shapley's orbit), an anticlockwise periodic orbit 
and a periodic orbit $\tilde \Gamma$ corresponding to the Jitter set $J$.
%These periodic orbits are linked, and lie in $\partial \Sigma$
%as in the figure below.
Moreover, 
\begin{itemize}
\item
for $\beta\in (0,\sigma)$ the clockwise periodic orbit  is attracting,
the anticlockwise orbit is of saddle-type, and the orbit $\tilde \Gamma$
is of jitter type;
\item 
for $\beta\in (\sigma,\tau)$,  the clockwise periodic orbit  is attracting,
the anticlockwise orbit is of saddle-type whereas the orbit $\tilde \Gamma$
is still of jitter type;
\item for $\beta\in (\tau,1]$,  the clockwise periodic orbit  is of saddle-type, 
the anticlockwise orbit is of attracting and the orbit $\tilde \Gamma$
is  of jitter type.
\end{itemize}
\end{prop}

That $\tilde \Gamma$ is a periodic orbit of `jitter type' 
means that it has the random walk behavior described in Theorem~\label{thm:inforbs}.

The induced flow is not smooth near the
periodic orbit $\tilde \Gamma$ corresponding to the Jitter set $J$.
Analyzing the local dynamics near $\Gamma$ is the main purpose of this paper
and is done in the final section of this paper.

\begin{proof}
The existence of the periodic orbits for
the induced flow follows immediately from the
existence of the corresponding orbits for the original system
(which were established in Sparrow et al \cite{SSH2008}).
In fact, the proof in the appendix in Sparrow et al \cite{SSH2008} shows that the induced
flow has a Shapley periodic orbit for $\beta\in [\sigma,1]$
even though the original flow does not (so for the original flow
the Shapley orbit spirals to $E$ when  $\beta\in [\sigma,1]$). 
Similarly the other two orbits exist for all $\beta\in (0,1]$.

So let us discuss the stability type.
If a periodic orbit of the original flow is
attracting (or repelling) then obviously the corresponding
periodic orbit of the induced flow  is also
attracting (repelling). If a periodic orbit $\gamma$
is originally of saddle-type then it depends on the
eigendirections: the  direction corresponding to the
(invariant) cone $C(\gamma)$ consisting of all rays from the
midpoint to the points on the periodic orbit $\gamma$
disappears in the induced flow. So we only need to 
consider the anticlockwise periodic orbit which we will still denote by $\gamma$.
In the appendix of Sparrow et al \cite{SSH2008} it was shown that
$\gamma$ consists of three line segments in $\Sigma$
and we computed the linear part of
the Poincar\'e transition map of $\gamma$ 
with sections taken in indifference planes.
These sections were taken
at two symmetrically positioned distinct points
computing in this way only the transition along one third
of $\gamma$. Because of the symmetry of the system, the actual return map to a section 
$Z$ is the third iterate of the linear map
computed in that appendix. Note that $Z$ 
was chosen to be contained in one of the indifference sets and so $Z$ contains $E$.
It was shown in Sparrow \cite{SSH2008} 
that  two eigenvalues are negative, and one positive (which was
equal to $n_2/n_1$). The positive eigenvalue
remains in $(0,1)$ for all $\beta\in (\sigma,1)$, one of the negative
eigenvalues remains in $(-1,0)$ for all $\beta\in (\sigma,1)$ whereas the
other negative eigenvalue is in $(-\infty,-1)$ for $\beta\in (\sigma,\tau)$
and in $(-1,0)$ for $\beta\in (\tau,1)$. For the induced
flow, the first return map has only two eigenvalues. 
Let us explain why the positive eigenvalue corresponds
to an eigenvector which lies in the cone $C(\gamma)$ and which 
therefore disappears after projecting to $\partial \Sigma$. 
Indeed, one of the eigenvalues of the linearisation at
$z:=Z\cap \gamma$ of the Poincar\'e map $P\colon Z\to Z$
lies in the cone $C(\gamma)$ (because this cone is invariant under the flow),
and so this eigenvalue is along the line segment $C(\gamma)\cap Z$
through $z$. This cone over the triangle $\gamma$ consists
of a surface made up of three (two-dimensional) triangles 
in $\Sigma$,  and so the corresponding eigenvalue
is positive (since the flow preserves $C(\gamma)$ and therefore an orbit
starting on one side of $\gamma\subset C(\gamma)$ remains on that side).
It follows that in the induced flow the
positive eigenvalue disappears under the projection.
The other two eigenvectors are projected, and their eigenvalues remain exactly the same
under the projection.  The original eigenvalue equation
now is solved modulo the direction corresponding
the projection, and so the two other eigenvalues stay the same for
the induced flow.  It follows that $\gamma$ is a saddle orbit if and only
if the corresponding orbit for the induced flow is a saddle orbit.
\end{proof}

\subsection{Many additional periodic orbits and chaos for the induced flow}

The next theorem is about the first return map to a section $Z$ based at some
point in $\tilde \Gamma$. It shows that orbits under this first return map
can move closer and further away from the fixed point. More precisely,
orbits can jump between the annuli $\{z\in Z; \dist(z,z_0)\in \left(\frac{1}{k_i+1},\frac{1}{k_i}\right)\}$
around $z_0$ in rather free way. Here $k_i$ corresponds to the number $N_i$ from 
Theorem~\ref{thm:dither}.  We refer to this behavior as of  `jitter type'.

\begin{theo}
\label{thm:inforbs}
Take $\beta\in (0,1)$ and consider the periodic orbit $\tilde \Gamma$ for the induced flow
corresponding to the jitter set $J$. Let $Z$ be a two-dimensional surface in $\partial \Sigma$
through some point $z_0\in \tilde \Gamma$ which is transversal to the induced flow,
and let $P$ be the Poincar\'e first return map of the induced flow to $Z$.  Then
\begin{itemize}
\item for each period $n\in \nz$,
there are infinitely many periodic orbits of $P$ 
of period $n$ (and one can even choose a sequence of  such periodic orbits 
so that the distance of the whole orbit to $z_0$ converges to zero);
\item there exist orbits which jitter in the following sense:
there exist a sum metric $\dist$ in $Z$ and $0<\lambda<1<\mu$ and $N_0$ so that 
for each sequence $k_i\in \nz$  so that
$$\lambda \le \dfrac{k_{i+1}}{k_i}\le \mu \mbox{ and }k_i\ge N_0,$$
there exists $z\in Z\setminus \Gamma$ with
$$\dist(P^i(z),z_0)\in (\frac{1}{k_i+2},\frac{1}{k_i})\mbox{ for all }i\ge 0.$$
\item the return map $P$ has subshifts of finite type, positive topological entropy 
 and has sensitive dependence on initial conditions. 
\end{itemize}
\end{theo}
\begin{proof}
We will give the proof of this theorem in Section~\ref{sec:proofoftheorem}.
\end{proof}

Denote the period of the periodic orbit $\tilde \Gamma$ under the induced flow by $T$. 
Since the induced flow is continuous, a periodic orbit $\tau_n$ which corresponds
to a period $n$ of the first return map $P$ and which is close to $\tilde  \Gamma$
has, under the induced flow, period approximately $n \cdot T$ (and the closer the orbit
is chosen to $\tilde \Gamma$ the better this approximation is).
Hence

\begin{cor}
Take $\beta\in (0,1)$ and consider the periodic orbit $\tilde \Gamma$ for the induced flow
corresponding to the jitter set $J$. Let $n\in \nz$ be arbitrary.
There exists a sequence of periodic orbits $\tau_k$, $k=1,2,\dots$ for the induced flow arbitrarily close to $\tilde \Gamma$ whose period converges to $nT$ as $k\to \infty$.
\end{cor}

%Similarly to the 2nd statement of the above theorem one can also find orbits which 'jitter'
%in prescribed ways. 
Furthermore,

\begin{cor}
Take $\beta\in (0,1)$.
The induced flow contains subshifts of finite type, positive topological entropy 
and has sensitive dependence on initial conditions. 
\end{cor}

\subsection{Global return section}
% near periodic orbit on $J$, and other dynamics on $\partial \Sigma$}

\begin{figure}[htp] \hfil
%\beginpicture
%\dimen0=0.2cm
%\setcoordinatesystem units <\dimen0,\dimen0>
%\setplotarea x from 0 to 30, y from 0 to 30
%\endpicture
\includegraphics[width=3in]{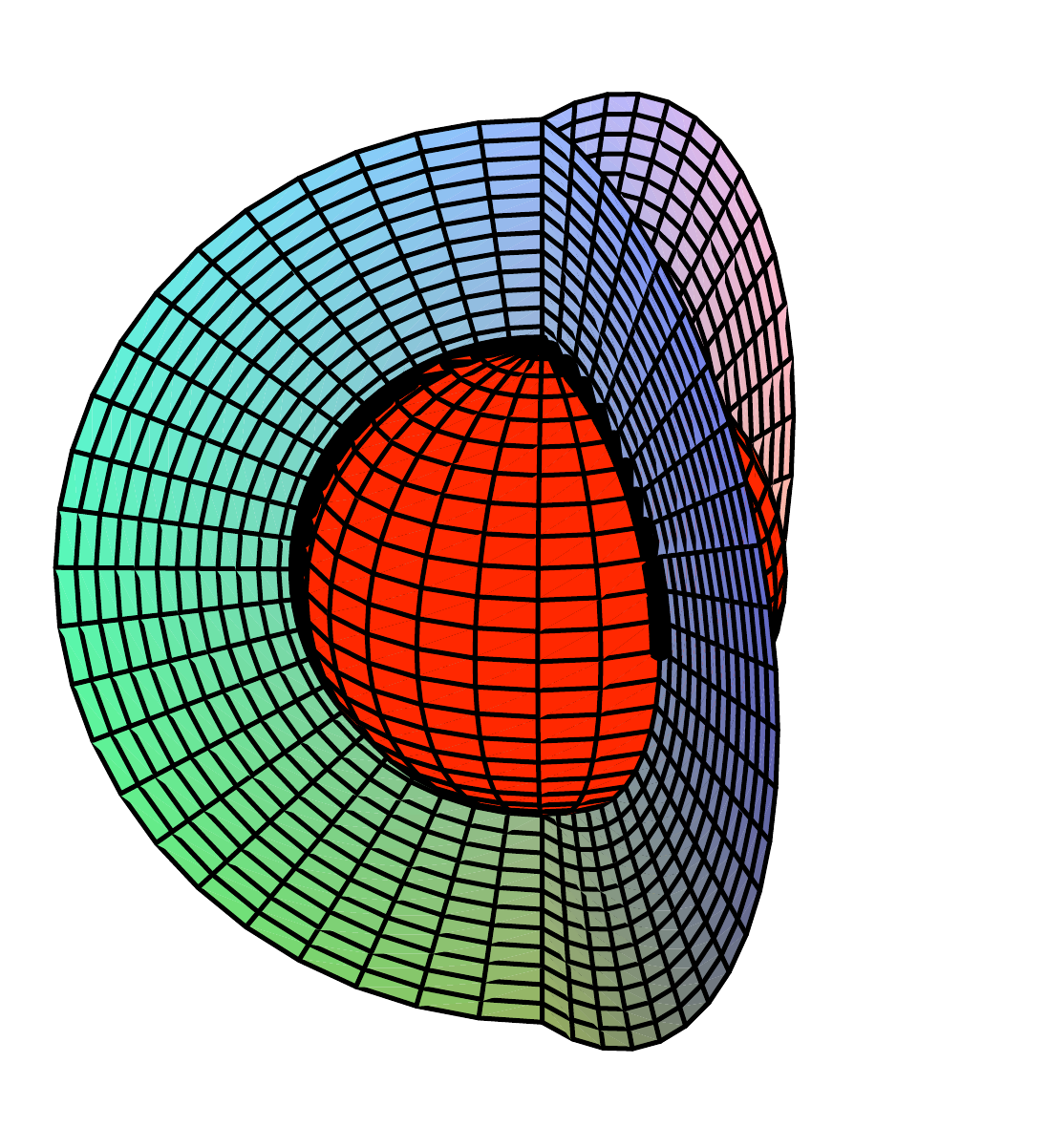}
\caption{\label{figs3}The subset of $\partial \Sigma$ where players $A$ and $B$
are indifferent, where we identify $\partial \Sigma$ with $\rz^3\cup \{\infty\}$.
Part of the periodic orbit $\tilde \Gamma$ is drawn as a fat curve. Another view
of this periodic orbit is in Figure~\ref{figs3gamma} (rotated 180 degrees about the $z$-axis).}
\end{figure}

\begin{figure}[htp] \hfil
%\beginpicture
%\dimen0=0.2cm
%\setcoordinatesystem units <\dimen0,\dimen0>
%\setplotarea x from 0 to 30, y from 0 to 30
%\endpicture
\includegraphics[width=3in]{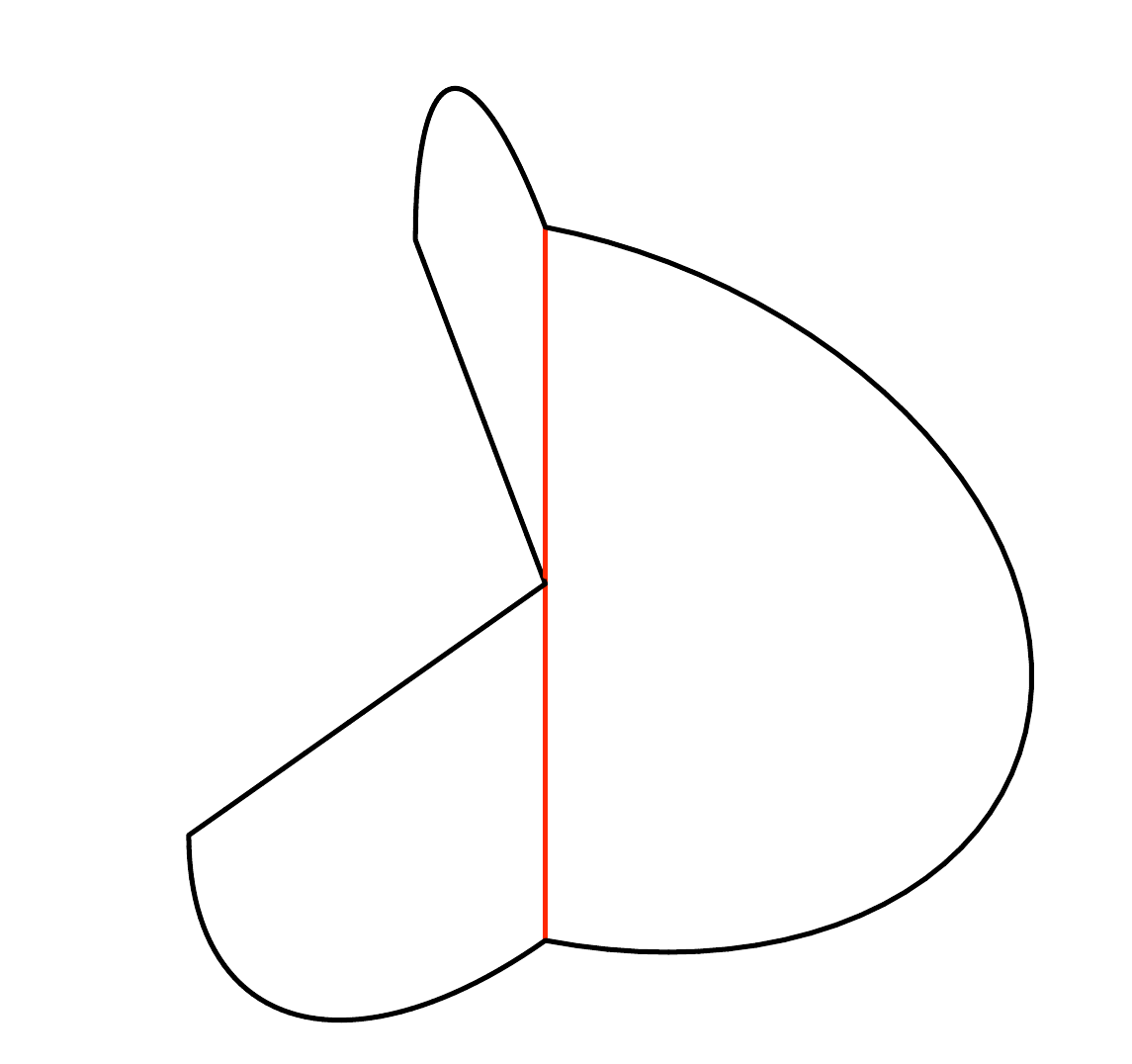}
\caption{\label{figs3gamma}The periodic orbit $\tilde \Gamma$ in the Jitter set drawn in $\partial \Sigma$ where we identify $\partial \Sigma$ with $\rz^3\cup \{\infty\}$. The vertical line is not part of the curve, but  is drawn for clarity and because it forms the part where the
four pieces from the global return section, introduced in Proposition~\ref{prop:section} meet
(the half-disc, and the two quarter discs, which meet along the vertical axis). The global return 
surface $S$ is the surface consisting of the half-disc, and two quarter discs
meeting at the z-axis. The return section $Z$ mentioned in Theorem~\ref{thm:inforbs} is 
transversal to $\tilde \Gamma$ and so is {\em not} contained in $S$. }
\end{figure}

Remember that $\Sigma$ is a ball in $\rz^4$
and $\partial \Sigma$ is homeomorphic to
$S^3$. This set can be thought of as $\rz^3\cup \{\infty\}$.
Usually it is not easy to easy to give a geometric
image of dynamics on $S^3$. 
But in fact, we are lucky. Instead of taking a section through a point $z\in \tilde  \Gamma$ transversal
to the periodic
orbit $\tilde \Gamma$, we can find a set $S$ which is topologically a disc,
such that $\partial S=\tilde \Gamma$, and such that
orbits cross $S\setminus \tilde \Gamma$ transversally.
This disc lies in the indifference sets.
The subset of $\partial \Sigma$
 where player $A$ is indifferent between strategies $i$ and $j$ is 
equal to $(\partial \Sigma_A\times Z_{ij}^A)\cup ( \Sigma_A\times (Z_{ij}^A\cap \partial \Sigma_B))$.
This two-dimensional set is equal to a triangular tube together with a triangle at one
end of the tube, in other words, it is homeomorphic to a topological disc.
The boundary of this disc corresponds to the triangle $\partial \Sigma_A\times E^B$.
Note that for each pair $i,j$ the boundary of this disc is the same.
So if we identify $\partial \Sigma$ with $\rz^3\cup \{\infty\}$, the 
set where player $A$ is indifferent between two strategies
can be thought of as the union of the upper and lower part of the
unit sphere in $\rz^3$ (i.e. $\{(x,y,z)\st x^2+y^2+z^2=1 \mbox{ and } z \ne 0\}$)
and the disc $\{(x,y,z) \st x^2+y^2\le 1\mbox{ and }z=0\}$.
Similarly, the set where $B$ is indifferent can then
be thought of as
$\{(x,y,z)\st (x,y)=(r\cos \phi, r\sin \phi) \mbox{ with }r\ge 0
\mbox{ and }\phi=2\pi/3,4\pi/3,0\} \cup \{\infty\}$.
The choice for $\phi$ represents which of the two strategies
are indifferent for $B$. Again this represents three discs
which all meet along the circle
$\{(x,y,z) \st x=y=0\}\cup \{\infty\}$ in $\rz^3\cup \{\infty\}$.
The orbit $\Gamma$ lies on the intersection of the
sets where $A$ and $B$ are indifferent and is drawn in Figure~\ref{figs3gamma}.
It turns out that one can find a subset $S$ of the space
where player $A$ is indifferent, such that $\partial S=\tilde \Gamma$
and for which the orbits go through $S$ transversally.

\begin{prop}
\label{prop:section}
There exists a topological disc $S$ in
$\partial\Sigma$ with the following properties
\begin{itemize}
\item $S$ is a piecewise linear;
\item $\partial S=\tilde \Gamma$;
\item each orbit of the induced flow (except $\tilde \Gamma$)
intersects $\tilde \Gamma$ transversally;
\item the Poincar\'e return map to $S$ is a well-defined
homeomorphism.
\end{itemize}
\end{prop}
\begin{proof}
Let $S$ consists of four pieces within
the linear indifference sets $\Sigma_A\times Z^A_{i,j}$.
These pieces are $U_3\times Z^A_{1,3}$, $U_3\times Z^A_{2,3}$,
$U_1\times Z^A_{1,2}$, $U_1\times Z^A_{3,1}$,
where $U_i$ is the part of $\Sigma_A$ where player $B$
prefers to head for corner $P_i^A$.
The flow is transversal to each of these four linear sets.
Also, inspection in the diagram of Figure 7 in Sparrow et al
\cite{SSH2008} 
shows that no orbit can miss these
sets (except if it is in $J$).
The section $S$ forms a disc made up from the semi-disc and the two-half semi-discs
depicted in Figure~\ref{figs3gamma}. \end{proof}

Instead as in Figure~\ref{figs3gamma}, we can also visualise the 
section $S$ as in Figure~\ref{fig:S-triangles}. Indeed, $S\subset \partial \Sigma$ 
is made up of the following regions:

\begin{enumerate}
\item  (the two triangles in 
$\Sigma_A$ where player $B$ heads for $1$ resp. $3$)  \, $\times$ \, (the point $R^B_{12}$); 
\item (the part of  $\partial \Sigma_A$  where player $B$ heads for $1$ resp. $3$) \, $\times$ \, 
(the segment $[R^B_{12},E^B]$ in $\Sigma_B$); 
\item (the part of $\partial \Sigma_A$  where player $B$ heads for $3$)
 \, $\times$ \, (the segment  $[R^B_{23},E^B]$ in $\Sigma_B$); 
\item (the triangle $\Sigma_A$  where player $B$ heads for $3$)
 \, $\times$ \, (the point  $R^B_{23}$); 
\item  (the part of $\partial \Sigma_A$  where player $B$ heads for $1$)
 \, $\times$ \, (the segment  $[R^B_{31},E^B]$ in $\Sigma_B$); 
\item (the  part of  $\Sigma_A$  where player $B$ heads for $1$)
 \, $\times$ \, (the point $R^B_{23}$).
\end{enumerate}
Using polar-like coordinates based at $E^A\in \Sigma_A$,
the region 1 can be represented as an isosceles triangle (with, say, angles
$80, 50,50$). Attaching region 2 then gives a similar triangle (which is the bigger triangle in 
Figure~\ref{fig:S-triangles}).
Joining all these regions together appropriately gives
that $S$ can also be represented as in  Figure~\ref{fig:S-triangles}. In this figure
we also show an orbit of the first return map for $\beta=\sigma$ (the zero-sum case).
The orbits, which tend to the Nash equilibrium in the full flow, do so in a rather chaotic fashion.

\begin{figure}[htp] \hfil
%\beginpicture
%\dimen0=0.2cm
%\setcoordinatesystem units <\dimen0,\dimen0>
%\setplotarea x from 0 to 30, y from 0 to 30
%\endpicture
\includegraphics[width=3in]{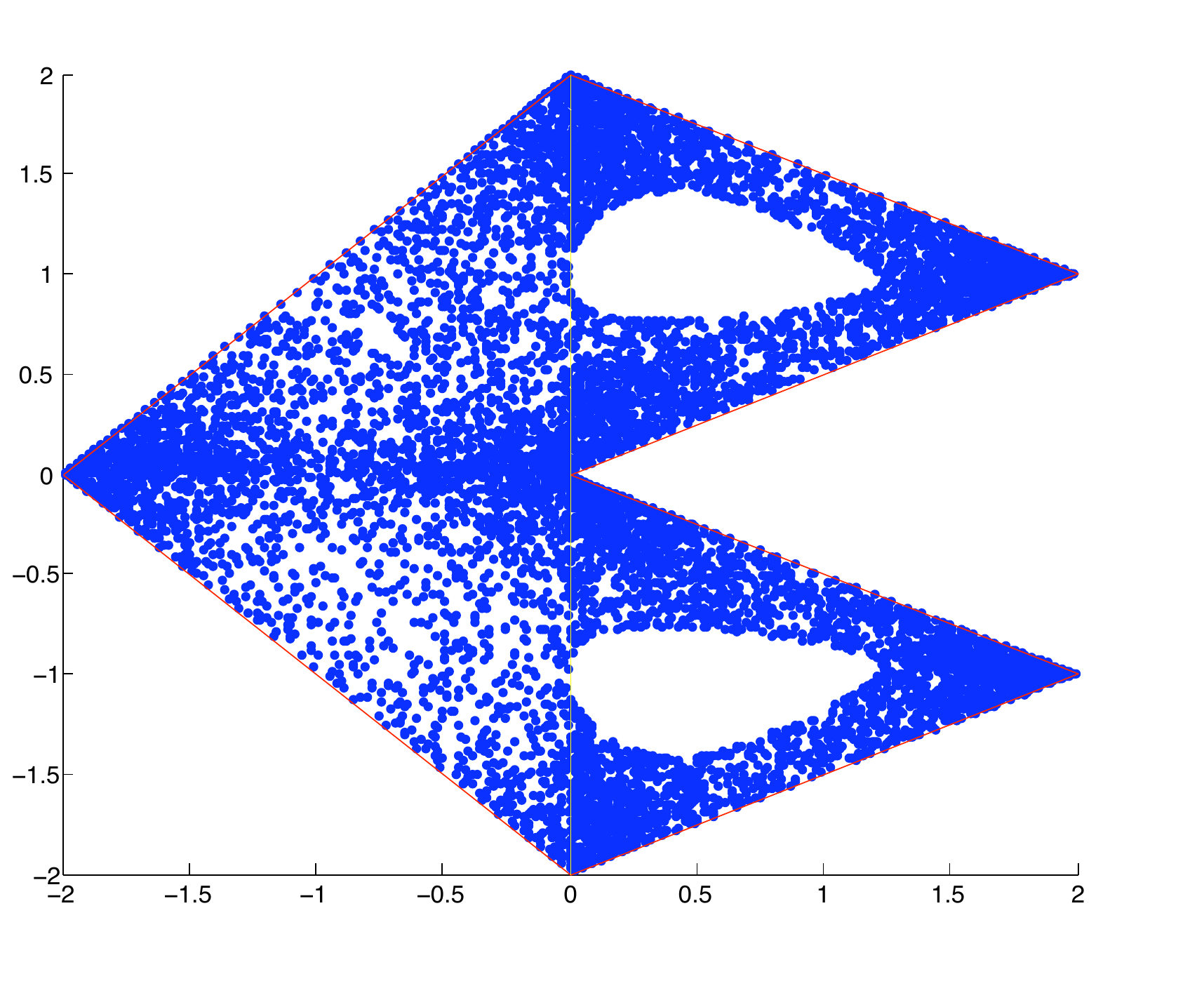}
\caption{\label{fig:S-triangles}The orbit of a typical starting point
under the first return map to $S$ 
when $\beta=\sigma$. The dynamics is similar to that of an area preserving map. There are
two 'egg-shaped' regions which are permuted by the first return map. Orbits within this
region form invariant circles.}
\end{figure}

%

%_________________________________SECTION---------------------------

\section{Dynamics for the original flow in $\Sigma$}
\label{s:dyninr4}

Let us now state the implications of the results
on the induced flow from the last section for the original flow.

\subsection{Invariant cones}
\label{ss:implics}

The implications of Theorem \ref{thm:inforbs} for the real
flow in $\Sigma_A \times \Sigma_B$ depend on the value of $\beta$,
as described in the following proposition.

%QQQ\tau \to \little \gamma
\begin{prop}\label{prop:implicationsigma}
For the original flow the statements from
Theorem \ref{thm:inforbs} still hold in the following sense.
\begin{itemize}
\item For $\beta \in (0,\sigma)$ each periodic orbit $\gamma$ near $\tilde \Gamma$ 
of the induced flow,
corresponds to an invariant cone $C(\gamma)$. Each orbit of the original 
flow starting in such a cone reaches the interior equilibrium in finite time 
(and then remains there).
\item  For $\beta \in (\sigma,1)$ periodic orbits near $\tilde \Gamma$ in $\partial \Sigma$
correspond to periodic orbits near $\Gamma$ on $J$. Thus there are
infinitely many orbits of the flow in this parameter range.
Moreover, there are orbits which jitter
in the sense of the 2nd assertion of Theorem \ref{thm:inforbs}
and the flow contains a subshift of finite type,
has  positive topological entropy and has sensitive dependence on initial conditions.
\end{itemize}
\label{prop:implics}
\end{prop}
\begin{proof}
In Prop A2 in Sparrow et al \cite{SSH2008} it was shown that each orbit
in $J$ tends to $E$ when $\beta\in (0,\sigma)$. In fact, on page 290 of that
paper it was shown that if we take a section in $J$, then 
orbits converge exponentially fast to zero under iterates of
the Poincar\'e map; moreover the time to spiral into the equilibrium $E$ is finite.

Now let $\tilde V$ be a plane in $\tilde \Sigma$
through a point $x_0\in \tilde \Gamma$ transversal 
to $\tilde \Gamma$ and $\tilde R$ be the first return map to $\tilde V$ corresponding
to the induced flow.
Identify $\Sigma$ with $\partial \Sigma\times [0,1]$
where $(x,0)\in \partial \Sigma\times [0,1]$ corresponds to $E$
and let $V=\tilde V\times [0,1]$. Let $R$ be the first return map to $V$
of the original flow. Then $R$ is of the form $R(x,t)=(\tilde R(x),\rho_x(t))$.
As we noted above, it was shown in Sparrow et al \cite{SSH2008} that 
$t\mapsto \rho_x(t)$ has derivative less than one when $x=x_0$ and 
$\beta\in (0,\sigma)$.
Note that
$$R^n(x,t)=(\tilde R^n(x),\rho_{\tilde R^{n-1}(x)}\circ \dots
\circ \rho_x (t))$$
Since for $\beta \in (0,\sigma)$,  $0$ is a hyperbolic 
attracting fixed point of $\rho_{x_0}\colon \rz^+\to \rz^+$, 
if $\tilde R^n(x)=x$, $x\in \partial \Sigma$ and $x,\dots,\tilde R^{n-1}(x)$ are all sufficiently close
to $x_0$ then $t\mapsto \rho_{R^{n-1}(x)}\circ \dots \rho_x (t)$
still has an attracting fixed point at $0$.
So the periodic point $x$ for the
induced flow corresponds to an orbit which tends towards
the equilibrium point $E$ as $t\to \infty$ for the original flow.
(Because of the parametrisation, orbits actually 
reach $E$ in finite time; during this time the orbit switches infinitely often between
strategies.)

On the other hand, if $\beta \in (\sigma,1)$, then there exists a 
$t_0\in (0,1)$ so that $\Gamma\cap V$ corresponds to $(x_0,t_0)$.
It was shown in Sparrow et al \cite{SSH2008} that $t_0$ is a hyperbolic attracting fixed point
(and $0$ a repelling fixed point)
of $\rho_z\colon \rz^+\to \rz^+$, where $t_0>0$ is so that
$\tilde z=(z,t_0)\in \Gamma$. So if the periodic point
$T^n(x)=x$ is sufficiently close
to $z$ then $t\mapsto \rho_{T^{n-1}(x)}\circ \dots \circ \rho_x (t)$
still has an attracting fixed point $t$ near $t_0$.
It follows that to each periodic point $x$ of the induced
flow is associated a periodic point $(x,t)$ of
the Poincar\'e return map of the flow (of the same period).
In the same way one can prove that there are
invariant sets which correspond to
the second and third assertions of Theorem~\ref{thm:inforbs}.
\end{proof}

%%%%%%%%%%%%%%%%%%%%%%%%%%%%%%%%%%%%%

\section{Proof of Theorem~\ref{thm:inforbs}}
\label{sec:proofoftheorem}
\subsection{A geometric description of the flow near the Jitter set}
\label{ss:newbit}

\begin{figure}[htp]
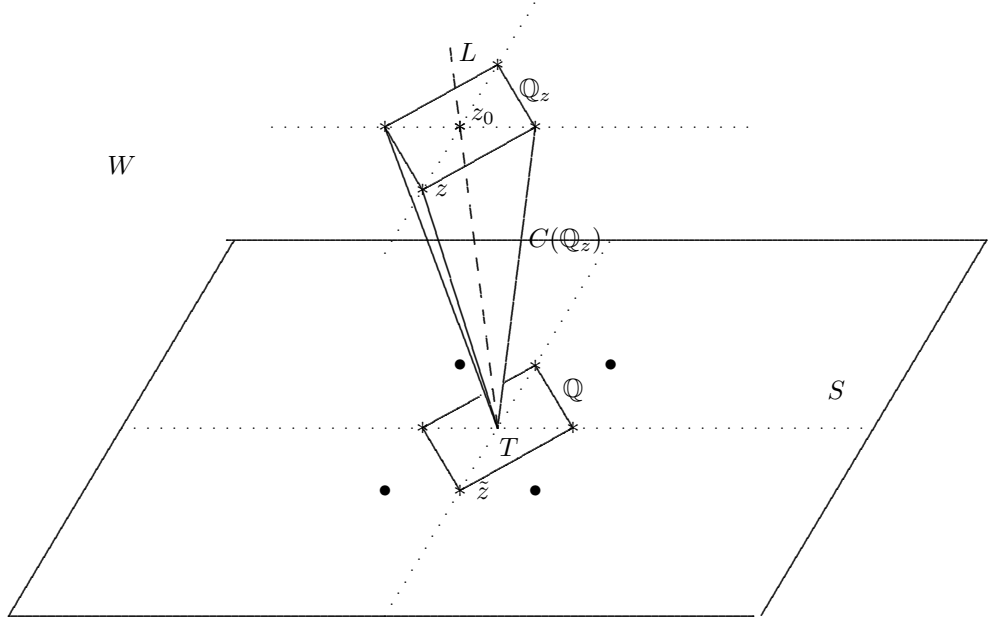
 \hfil
\beginpicture
\dimen0=0.5cm
\setcoordinatesystem units <\dimen0,\dimen0>
\setplotarea x from 0 to 25, y from 0 to 17
\plot 0 0 20 0 26 10 6 10 0 0 /
\put {$S$} at 22 6 
\setdots 
\plot 3 5 23 5 /
\plot 10 0 16 10 /
\setsolid
\plot 13 5 11 11.33 14 13 13 5 10 13 11 11.33 /
\plot 10 13 13 14.66 14 13 /
\setdots
\plot 7 13 20 13 /
\plot 10 9.66 14 16.33 /
\multiput {$*$} at 12 13 10 13 11 11.33 14 13 12 13 13 14.66 /
\put {$z_0$} at 12.6 13.3 
\put {$z$} at 11.5 11.33
\put {$\Q_z$} at 14 14 
\put {$W$} at 3 12 
\put {$L$} at 12.2 15 
\multiput {$*$} at 12 3.33 15 5 14 6.66 11 5 /
\put {$\tilde z$} at 12.6 3.33
\plot 12 3.33 15 5 14 6.66 11 5 12 3.33 /
\setsolid
\plot 12 3.33 15 5 14 6.66 13.2 6.22 /
\plot 12.5 5.82 11 5 12 3.33 /
\put {$\Q$} at 15 6 
\put {$C(\Q_z)$} at 14.8 10 
\multiput {$\bullet$} at 10 3.33 14 3.33 12 6.66 16 6.66 /
\put {$T$} at 13.3 4.5
\setdashes
\plot 13 5 12 13 11.7 15.4 /
\endpicture
\caption{\label{fig:quad}{\small
The targets $P^A_i,P^A_j,P^B_k,P^B_l$ in the plane
$S$ are drawn marked as $\bullet$. The dashed line corresponds to
$Z$. A quadrangle $\Q$ in $S$ is also drawn
and also another quadrangle $\Q_z=\mbox{pr}^{-1}(\Q)\cap S_0$ which is in
the plane $S_0\ni z_0$.
The cone of the quadrangle with apex $T$ is invariant under the flow.
In the figure we have taken the case that $d_i^a$ and $d_i^b$
are equal to $1$.}}
\end{figure}

Before doing a rather cumbersome explicit calculation in the next subsection
(for the game under consideration), 
we first want to explain geometrically what the dynamics near $Z^*$ looks like.
As before let $Z'=Z^B_{k,l}\times Z^A_{i,j}$
be a codimension-two plane where both players
are indifferent between two strategies. In this section we consider
the situation that near  part of this set where both players are indifferent,
both players choose repeatedly 
the strategies in a period four pattern $(i,k), (i,l), (j,l), (j,k)$.
%$$\begin{array}{l}	
%\mbox{strategy player A \quad }  
%         i     i  j  j\\%  \dots  i   i  j  j \\
%\mbox{strategy player B \quad }  
%\overl k     l  l  k\end{array}%  \dots k l  l k           .  \end{array}
%$$
Let $S$ be the two-dimensional
plane spanned by 
$ [P^A_i,P^A_j]\times [P^B_k,P^B_l]\subset \Sigma_A\times \Sigma_B$.
The orbit segments aim to the four targets in $S$, but an orbit 
which starts in $Z'$ remains in $Z'$ and aims for the point $T:=S\cap Z'$. 
We call this point the {\em cone-target}.
The linear (affine) spaces $S$ and $Z'$ are of complementary
dimensions and transverse. Let $L$ be a line through the cone-target 
$T$ contained in $Z'$, and
let $W$ be the three dimensional space $W=S + L$.
Take $z_0\in L$, and assume that on some neighbourhood
$U\subset W$ of $z_0$, the players only choose the above strategies. 
In other words, $U\setminus Z'$ consists of four components, 
$S_{st}$ where $s\in \{i,j\}$ and $t\in \{k,l\}$ such that
players $A$ and $B$ aim for $P_s^A$ resp. $P_t^B$ in $S_{st}$.
First we prove that each orbit in $U$
lies within a cone with apex $T$ (the cone-target) induced over some quadrangles $\Q$, see Figure~\ref{fig:quad}.

Let us define these quadrangles $\Q$.
To do this, it is convenient apply a translation to $S\subset \Sigma_A\times \Sigma_B$ 
which puts  $T=S\cap Z'$ as the origin $(0,0)$ of 
$S\subset \rz^2$ and identifies
$S=[P_{i}^A,P_j^A]\times [P_k^B,P_l^B]$ with a rectangle in $\rz\times \rz$ so that  
$(P^A_i,P^B_k)$, $(P^A_i,P^B_l)$,
$(P^A_j,P^B_l)$ and $(P^A_j,P^B_k)$ correspond to
$(-d^a_1,d^b_1),(-d^a_1,-d^b_2),(d^a_2,-d^b_2),(d^a_2,d^b_1)$ 
for some $d^a_1,d^a_2,d^b_1,d^b_2\ne 0$
with $d^a_1 d^a_2>0$ and $d^b_1 d^b_2>0$.
(The signs of $d^a_1 ,d^a_2,d^b_1, d^b_2$ depend on whether
the orbits flow clockwise or anticlockwise along $L$.)
% or with $(-d^a_1,-d^b_1),(-d^a_1,d^b_2),(d^a_2,d^b_2),(d^a_2,-d^b_1)$ 
Now let $\Q$ be the quadrangle with corners
\begin{equation}
(d^a_2,0),(0,\frac{d^a_2}{d^a_1}d^b_1 ),
(- \frac{d^b_1}{d^b_2}d^a_2,0),(0,-d^b_1)
\label{qcc}
\end{equation}
(or a multiple of it).
Note that the four corners of this quadrangle lie on the coordinate axes
and the sides of this quadrangle are
parallel with the vector pointing from $O$ to $(P_s^A,P_t^B)$ 
in the region $\hat S_{st}$ corresponding to the projection of $S_{st}$ along the
direction $L$. The quadrangle $\Q$ is shown in Figure~\ref{fig:quad}. 

Next take a plane $S_0\subset W$ through $z_0$ transversal to $L$
and take the projection $\mbox{pr}\colon S_0\to S$ along $L$. 
Furthermore, take $z\in S_0$, take the multiple $\epsilon \Q$ 
of the quadrangle $\Q$ in $S$
which contains $\mbox{pr}(z)\in S$ and define a quadrangle 
$\Q_z =\mbox{pr}^{-1}(\epsilon \Q)\cap S_0$, see Figure~\ref{fig:quad}.
Each $z\in S_0$ sufficiently close to $z_0$ is contained in some quadrangle $\Q_z$
in this way. 
These quadrangles in $S_0$ have the following two property: (1)
their corners are contained in two lines in $S_0$ which are orthogonal to each other
and (2) the quadrangles  are self-similar (they are all scalings
of each other around the common `centre point' $z_0$).

The reason these quadrangles are important is the following:
Consider the cone $C(\Q_z)$ over $\Q_z$ with as apex the cone-target $T$. 
Since the 'vertical' sides of this cone are contained in planes through the cone-target
$T$ and the targets $(P^A_s,P^B_t)$, for each $z$ in this side the vector
pointing from $z$ to $(P^A_s,P^B_t)$ is contained in the side of the cone.
It follows that orbits which start in $C(\Q_z)$ remain in $C(\Q_z)$ until such time
as one or both of the players start to play a third strategy.

\begin{figure}[htp]
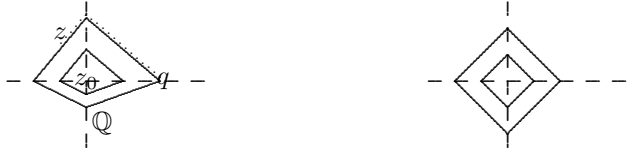
 \hfil
\beginpicture
\dimen0=0.07cm \dimen1=0.07cm
\setcoordinatesystem units <\dimen0,\dimen0> point at 0 0
\setplotarea x from -10 to 10, y from -15 to 15
\setlinear \setdashes
\plot -15 0 25 0 /
\plot 0 15 0 -15 / 
\put {$z_0$} at 0 0 
\setsolid
\put {$q$} at 14.6 0
\put {$z$} at -5 9
\put {$\Q$} at 3 -8  
\plot -10 0  0 12  14 0 0 -5 -10 0 / 
\plot -5 0  0 6  7 0 0 -2.5 -5 0 / 
\setdots <1mm>
\plot 14 0.7 0 12.7 -5 7 /
\setsolid
\setcoordinatesystem units <\dimen0,\dimen0> point at -80 0
\setplotarea x from -10 to 10, y from -10 to 10
\setlinear \setdashes
\plot -15 0 25 0 /
\plot 0 15 0 -15 / 
\setsolid
\plot -10 0 0 10 10 0 0 -10 -10 0 /
\plot -5 0 0 5 5 0 0 -5 -5 0 /
\endpicture
\caption{\label{fig:quadrangle}
{\small Arbitrary quadrangles can be mapped to standard quadrangles
$\{(x,y)\in \rz^2 ; |x|+|y|=r\}$ by piecewise linear maps. An arc $\gamma$ connecting
one of the corners of $\Q$ to some point $z\in \Q$ is also drawn (with dotted points). 
The angle of $z$ is equal to
the $l(\gamma)/l(\Q)$ where $l$ is the usual Euclidean length in $\rz^2$.}}
\end{figure}

So consider such a family $\Q$ of quadrangles. Let us
define a natural metric in $S_0$ associated to these quadrangles
in the following way. The family $\Q$ of quadrangles can be mapped to the standard family of quadrangles 
$\{(x,y)\in \rz^2 ; |x|+|y|=r\}$ 
by a map $L$ which restricted to each quadrant is linear, see Figure~\ref{fig:quadrangle}
and with $L(z_0)=0$.
Thus we can define for each $z\in S_0$,
$$||z||_{S_0}=||Lz||$$
where $|| \cdot ||$ is the sum-norm on $\rz^2$:  if $w=(x,y)$  then  $||w||=|x|+|y|$.
So $||z_0||_{S_0}=0$ and the set 
$\{w\in S_0 ; ||w||_{S_0}=r\}$ is exactly a quadrangle from the above family.
By analogy to the usual polar coordinates, we can associate
an angle to $z\in S_0\setminus \{z_0\}$ in the following way. Pick one of the corners
$q$
of the quadrangle $\Q=\{w\in S_0 ; ||w||_{S_0}=r\}$ where $r=||z||_{S_0}$, let 
$\gamma$ be the curve on this quadrangle which connects $q$ to $z$ 
(anti-clockwise). Then define $\phi(z)=2\pi \dfrac{l(\gamma)}{l(\Q)}$ where $l$ stands for the usual 
Euclidean length in $\rz^2$.
Thus we have defined 
{\em quadrilateral polar coordinates} 
 $(r,\phi)$ of $z\in S_0\setminus \{z_0\}$ as follows:
$$r:=||z||_{S_0} \mbox{ and }\phi=2\pi \dfrac{l(\gamma)}{l(\Q)}$$
where $\Q=\{w\in S_0 ; ||w||_{S_0}=r\}$ and $r=||z||_{S_0}$.
(This is completely analogous to how the usual polar coordinates are defined.)

\begin{prop}[Poincar\'e transition map for two planes parallel to $S$.]
\label{spiralest}
Take two points $z_0,z_0'\in L$
so that along the segment $[z_0,z_0']$ no other strategy
becomes preferential (or indifferent) to the strategies
$i,j$ for $A$ and $k,l$ for $B$.
Let $S_0,S_0'$ be  two-dimensional
planes in $W$ through $z_0$
resp. $z_0'$ which are both transversal to $L$. 
For $z\in S_0$ consider the quadrangle $\Q_z$ in $S_0$ constructed 
above and let $C(\Q_z)$ be the cone of $\Q_z$ over the cone-target $T$.
Let $R(z)$ be the Poincar\'e map from $S_0$ to $S_0'$. 
Then $R(z)$ is well-defined for $z$ close to $z_0$
and is contained in $C(\Q_z)\cap S_0'$.
Moreover, consider polar coordinates
in the plane $S_0'$ with the distance and angle taken from
the point $S_0'\cap L$. 
Then one can take a {\em continuous}  map 
$$\Psi\colon S_0 \setminus \{z_0\} \ni z \mapsto \rz$$
so that for each $z\in S_0\setminus \{z_0\}$
the value of $\Psi(z)$ modulo $2\pi$
is equal to the quadrilateral polar angle of $R(z)\in S_0\setminus \{z_0'\}$ 
(as defined above).
Then
$\Psi(z)$ is equal to 
\begin{equation}
2\pi\cdot  \frac{1-c(z)}{a\cdot c(z) \cdot r} + B(z) +B_0. 
\label{eq:winding}\end{equation}
Here $r:=||z||_{S_0}$,
$c(z)$ is equal to $c_0(1+O(r))$ where 
$c_0=\dist(z_0,T)/\dist(z_0',T) \in (0,1)$ and where $O(r)$
is a function with $O(r)/r$ bounded as $r\to 0$,
\,\, $B(z)\colon S_0\setminus \{z_0\} \to \rz$
is continuous function on $ S_0$ with
$B(z)\to 0$ as $z\to z_0$, and $B_0\in \rz$ and $a>0$ are constants. Here $\dist $ is the usual Euclidean
norm on the line $L$. 
\end{prop}

So the angle of $R(z)$ increases very fast as $r=||z||_{S_0}$ tends to zero.

%\begin{prop}[Poincar\'e transition map for two planes transversal to $l$.]
%\label{spiralest2}
%Take the notation from above.
%Let $\tilde S_{q'}$ be a two-dimensional plane
%through $q'$ transversal to $l$
%and contained in $<S,l>$ (i.e. in the three dimensional space
%spanned by $S$ and $l$).
%Identify $\tilde S_{q'}$ with $S$ by linear projection along 
%(i.e. parallel to) $l$ and
%let $\tilde T$ be the transition map from $S_q$ to $\tilde S_{q'}$.
%Then we have
%$$\tilde T(x)=c\cdot \tilde L^{-1} \circ \R_{2\pi\cdot
% \frac{1-c}{a\cdot c \cdot r} + \tilde B(x)}
%\circ L(x)$$
%where $c, L, \tilde B$ are as in the proposition above
%and $\tilde L\colon \rz^2\to \rz^2$ is a diffeomorphism
%with $\tilde L(0)$ and $D\tilde L=L$.
%\end{prop}

%\bigskip

%Let us first prove Proposition~\ref{spiralest}.

%\bigskip

\begin{figure}[htp]
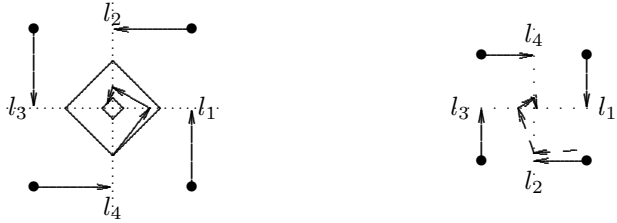
 \hfil
\beginpicture
\dimen0=0.07cm \dimen1=0.07cm
\setcoordinatesystem units <\dimen0,\dimen0> point at 0 0
\setplotarea x from -10 to 10, y from -10 to 10
\setlinear \setdots <1mm>
\plot -20 0 20 0 /
\plot 0 20 0 -20 / 
\put {$l_1$} at 18 0
\put {$l_2$} at 0 18
\put {$l_3$} at -18 0 
\put {$l_4$} at 0 -19 
 \setsolid
%\put {$(2)$} at 0 -14
\put {$\bullet$} at 15 15
\put {$\bullet$} at -15 15
\put {$\bullet$} at -15 -15
\put {$\bullet$} at 15 -15
\arrow <5pt> [0.2,0.4] from 15 15 to 0.5 15
\arrow <5pt> [0.2,0.4] from -15 15 to -15 0.5
\arrow <5pt> [0.2,0.4] from -15 -15 to -0.5 -15
\arrow <5pt> [0.2,0.4] from 15 -15 to 15 -0.5
\setsolid %\arrow <5pt> [0.2,0.4] from -8 -8 to 0 -8.5
\arrow <5pt> [0.2,0.4] from 0 -9 to 7 0
\arrow <5pt> [0.2,0.4] from 7 0 to 0 4
\arrow <5pt> [0.2,0.4] from 0 4 to -1 1
\setsolid
\plot 2 0  0 2   -2 0  0 -2  2 0 / 
\plot 9 0  0 9   -9 0  0 -9  9 0 / 
\setsolid
\setcoordinatesystem units <\dimen0,\dimen0> point at -80 0
\setplotarea x from -10 to 10, y from -10 to 10
\setlinear \setdots
\plot -15 0 15 0 /
\plot 0 15 0 -15 / \setsolid
\put {$l_1$} at 14 0
\put {$l_2$} at 0 -14
\put {$l_3$} at -14 0 
\put {$l_4$} at 0 14
%\put {$(2)$} at 0 -14
\put {$\bullet$} at 10 10
\put {$\bullet$} at -10 10
\put {$\bullet$} at -10 -10
\put {$\bullet$} at 10 -10
\arrow <5pt> [0.2,0.4] from 10 10 to 10 0.3
\arrow <5pt> [0.2,0.4] from -10 10 to -0.3 10
\arrow <5pt> [0.2,0.4] from -10 -10 to -10 -0.3
\arrow <5pt> [0.2,0.4] from 10 -10 to 0.3 -10
\setdashes
\arrow <5pt> [0.2,0.4] from 8 -8 to 0 -8.5
\arrow <5pt> [0.2,0.4] from 0 -8.5 to -3 0
\arrow <5pt> [0.2,0.4] from -3 0 to 0 2
\arrow <5pt> [0.2,0.4] from 0 2 to 0.5 0
\setsolid
\endpicture
\caption{\label{fig:2x2}
{\small The spiral motion in $S$ and the half-lines $l_i$ emanating from the centre $O$.
The transition map from some quadrangle $\Q$ to $c_0\Q$ is also drawn.}}
\end{figure}

\noindent
\begin{proof} %{\em Proof of Proposition~\ref{spiralest}:}
The only thing we need to prove is that (\ref{eq:winding}) holds.
To do this, note that the vector field is the product of
a vector field in a direction along $L$
and one in the direction parallel to $S$. If we ignore the
direction along $L$, then we get a new two-dimensional vector field on $S$
which corresponds to a two-dimensional game with spiral behaviour, see Figure~\ref{fig:2x2}.
In other words, if we define $\mbox{pr}\colon W\to S$ to be the (linear) projection
along $L$ and let $(z,t)\mapsto \Psi_t(z)$ be the flow through
$z$, then $\mbox{pr}$ projects the orbits of this flow in $W$
to orbits of a two-dimensional system in $S$.
Moreover $\Q_{z}=\mbox{pr}(\Q_{\tilde z})$ is a quadrangle in $S$ as defined above,
where $\tilde z:=\mbox{pr}(z) \in S$.
Denote the flow of this two-dimensional system through $\tilde z$
 by $(\tilde z,t) \mapsto \tilde \Phi_t(\tilde z)$.
Consider the piece $[0,A]\ni t \mapsto \Phi_t(z)$ 
of the orbit through $z\in S_0$ until  it hits $S_0'$.
Then $[0,A]\ni t\mapsto \tilde \Phi_t(\tilde z)$ is an arc with
$\tilde z\in \Q_{\tilde z}$ and $\tilde R(\tilde z)\in c(z) \cdot \Q_{\tilde z} $
where $c(z)>0$ is equal to $c_0(1+O(\dist(z,z_0)))$
where 
$$c_0:=\dist(S_0'\cap L,O)/\dist(S_0\cap L,O)$$
and $O(x)$ is a function so that $O(x)/x$ is bounded when $x\to 0$.
This holds because the angle between the 'vertical sides' of $C(\Q_z)$ and  $L$
tends to zero as $z\in S$ tends to $z_0$. (In fact, if $S_0$ and $S_0'$ are both
parallel to $S$ then $c(z)$ is exactly $c_0$.)

So it suffices to consider the projected flow $\tilde \Phi_t(\tilde z)$ on $S$. 
That is, let us consider the Poincar\'e transition 
map $\tilde R$ which assigns to a point $\tilde z\in \Q_{\tilde z}\subset S$ 
the intersection of $c_0\Q_{\tilde z}$ with the orbit $t\mapsto \tilde \Phi_t(\tilde z)$.
Next consider the quadrilateral polar coordinates in $S$ (with the origin centered at $O$)
and let $t\mapsto \phi_t(\tilde z)$ be the angle of $\tilde \Phi_t(\tilde z)$
(where we choose $t\mapsto \phi(\tilde z)$ continuous).
Then the angular change $\phi_A(\tilde z)-\phi_0(\tilde z)$ is equal to
the $2\pi$ times the integer number of times the  $[0,A]\ni t\mapsto \tilde \Phi_t(\tilde z)$
winds around $O$ plus some number in $[0,2\pi]$.
To compute this integer number of winding, consider the four 
half-line $l_i$ through the equilibrium $E_S$ of the two-dimensional game
in $S$ where one of the players 
is indifferent and the other player always prefers one strategy,
and denote these by $l_m$, $m=1,\dots,4$
(numbered so the flow meets these half-lines periodically in this order).
Given $x\in l_m$, let $f_m(x)$ be the first time the flow meets $l_{m+1 \mod{4}}$.
Let us identify these lines with $[0,1]$ where $0$ corresponds
to the equilibrium $E_S$.  Since going from $l_m$ to $l_{m+1}$
is just the stereographic projection from $l_m$ to $l_{m+1}$
through lines of the target in the next region, $f_m(x)$ is of the form
$$r\mapsto \frac{r}{1+ a_m r} \mbox{ with }m=1,2,3,4$$
where $r$ stands for the distance to the origin measured in the metric $||\cdot ||_{S}$
(if, instead, we take the Euclidean metric then we get  these maps are of the form
$r\mapsto \frac{\theta_m r}{1+ a_m r}$ where $\theta_m$ is related to the
shape of the quadrangle). Since the composition of
two Moebius transformations of the form $r\mapsto r/(1+\rho_1 r)$
and $r\mapsto r/(1+\rho_2 r)$ is equal to $x\mapsto r/(1+(\rho_1+\rho_2)r$,
we get  that the Poincar\'e first return map $f$ to $l_1$ is of the form
$$f(r)=f_4\circ f_3\circ f_2\circ f_1 (r)= \frac{r}{1+a r}\mbox{ where }a=a_1+a_2+a_3+a_4.$$
Hence $f^n(r)=\frac{r}{1+nar}$ for all $n$.
So let $n$ be maximal so that $f^n(r)\ge c r$ where $c$ is equal to the number
$c(z)$  from above. Then $n$ is the maximal integer so that
$r/(1+nar)\ge c\cdot r$, i.e.,
$$n\le \frac{1}{ar}\left( \frac{1-c}{c}\right).$$

In fact, $\phi(\tilde R(\tilde z))-\phi(\tilde z)$ is constant for $z\in \Q$,
because the relative length of an arc in $\Q$ (as a proportion of
total perimeter length of $\Q$) is preserved under one central projections, see
Figure~\ref{fig:2x2b} and therefore also under a composition of such maps. 
So in particular if $c_0$ is equal to $1/(1+nar)$ then 
the angle of $\tilde R(\tilde z)$ and $\tilde z$ are the same modulo $2\pi$. 
If $1/(1+(n+1)car) < c_0 <  1/(1+ncar)$ then the result
follows from a simple geometric consideration, see Figure~\ref{fig:2x2b}. 
Thus we have proved Proposition~\ref{spiralest}.
\end{proof}

\begin{figure}[htp]
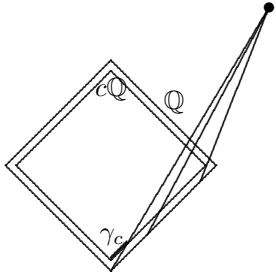
 \hfil
\beginpicture
\dimen0=0.07cm \dimen1=0.07cm
\setcoordinatesystem units <\dimen0,\dimen0> point at 0 0
\setplotarea x from -10 to 10, y from -20 to 20
\setlinear 
\plot -20 0 0 20 20 0 0 -20 -20 0 /
\put {$c\Q$} at 0 14.5 
\put {$\Q$} at 12 12 
\put {$\bullet$} at 30 30
\plot 0 -17.4 3 -14.4 /
\plot 0 -17.6 3 -14.6 /
\put {$\gamma_c$} at 0.2 -13.5
\plot 0 -20 30 30  /
\plot 7 -13 30 30 /
\plot 17 -3 30 30 /
\setsolid
\plot  -18 0 0 18 18 0 0 -18 -18 0 /
\setsolid
\endpicture
\caption{\label{fig:2x2b}
{\small Relative length (i.e. quadrilateral angle) is preserved by central projection. 
Moreover, let $\Q=\{z; ||z||_S=r\}$. Then the length of the arc
$\gamma_c$ between the vertical axis and the intersection of the line with $c\Q$ is equal to
$(1-c)\theta r$ whereas the length of $\Q$ is equal to $\sqrt{2}cr$, where $\theta$ is a constant. 
So the angle of $\gamma_c$ is equal to $(1-c)\theta /(\sqrt{2}c)$ (plus a constant). }}
\end{figure} 

Note that after $n=[1/r]$ iterates of the map
$f$ one has
$f^{n}(r)=\frac{r}{1+nar}\approx c_1 r$
where $c_1\in (0,1)$ is equal to $1/(1+a)$.
During this time, the orbit has spiraled $n$ times
around $O$ with each spiral between $T_r$ and $T_{c_1 r}$.
The length (and so the time-length)
of the two-dimensional orbit is roughly $n \cdot c_2 r\approx c_2$.
Hence in one unit of time, the flow moves
a point $\tilde z$ a definite factor closer to $O$.

\subsection{An analytic computation of the flow near the Jitter set}
\label{subsec:analyticcomp}

\begin{figure}[htp]
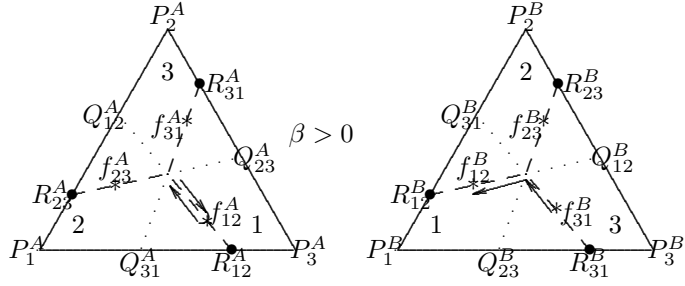
 \hfil
\beginpicture
\dimen0=0.17cm
\setcoordinatesystem units <\dimen0,\dimen0>
\setplotarea x from 0 to 30, y from -2 to 20
\put {$\beta>0$} at 22 9
\setlinear
\setsolid
\plot 0 0 20 0 /
\plot 0 0 10 17.3 20 0 /
\setdashes
\plot 15 0 10 5.98 /
\plot 2.5 4.33 10 5.98 /
\plot 12.5 13 10 5.98 /
\setdots
\plot 10 5.98 6.14 10.60 /
\plot 10 5.98 15.80 7.26 /
\plot 10 5.98 7.88 0 /
\setsolid
\multiput {*} at 11.5 9.8 /  %11.4 9.8 /
\multiput {*} at 13.1 2.0 5.9 4.8 /
\multiput {$\bullet$} at 15 0 2.5 4.33 12.5 13 /
\put {$f_{12}^A$} at 14.5 3
\put {$f_{31}^A$} at 10 9.8
\put {$f_{23}^A$} at 5.9 6.5
%\put {$\alpha^A$} at 13 1       % added
%\put {$\alpha^A+\pi/3$} at 17.4 1 % added
%\put {$\frac{2\pi}{3}-\alpha^A$} at 5 9 % added
\put {$R_{12}^A$} at 15 -1
\put {$R_{31}^A$} at 14.5 13
\put {$R_{23}^A$} at 1 4.33
\put {$Q^A_{12}$} at 5 10.60
\put {$Q^A_{23}$} at 16.80 7.26
\put {$Q^A_{31}$} at 7.88 -1
\put {$P^A_1$} at -1 0
\put {$P^A_3$} at 21 0
\put {$P^A_2$} at 10 18.3
\put {$2$} at  3 2
\put {$3$} at 10 14
\put {$1$} at 17 2
\arrow <5pt>  [0.2,0.4] from 11 5.5 to 13 3
\arrow <5pt>  [0.2,0.4] from 12.3 2.3 to 10.2 5.0
\circulararc -180 degrees from 13 3 center at 12.7  2.6
%\setdots <0.3mm>
%\plot 16.5 0.2 14 0.2 12 6.2 / 
\setcoordinatesystem units <\dimen0,\dimen0> point at -28 0
\setplotarea x from 0 to 30, y from -5 to 20
\setsolid
\plot 0 0 20 0 /
\plot 0 0 10 17.3 20 0 /
\setdashes
\plot 15 0 10 5.98 /
\plot 2.5 4.33 10 5.98 /
\plot 12.5 13 10 5.98 /
\setdots
\plot 10 5.98 6.14 10.60 /
\plot 10 5.98 15.80 7.26 /
\plot 10 5.98 7.88 0 /
\setsolid
\multiput {*} at 11.4 9.8 12.4 3 5.9 4.9 /
\multiput {$\bullet$} at 15 0 2.5 4.33 12.5 13 /
\put {$f_{31}^B$} at 14.0 3
\put {$f_{23}^B$} at 10 9.8
\put {$f_{12}^B$} at 5.9 6.5
\put {$R_{31}^B$} at 15 -1
%\put {$\alpha^B$} at 13 1       % added
%\put {$\alpha^B+\pi/3$} at 17.4 1 % added
%\put {$\frac{2\pi}{3}-\alpha^B$} at 5 9 % added
\put {$R_{23}^B$} at 14.5 13
\put {$R_{12}^B$} at 1 4.33
\put {$Q^B_{31}$} at 5 10.60
\put {$Q^B_{12}$} at 16.80 7.26 %%% WAS WRONG SSSSS
\put {$Q^B_{23}$} at 7.88 -1 %%% WAS WRONG SSSSS
\put {$P^B_1$} at -1 0
\put {$P^B_3$} at 21 0
\put {$P^B_2$} at 10 18.3
\put {$1$} at  3 2
\put {$2$} at 10 14
\put {$3$} at 17 2
\arrow <5pt>  [0.2,0.4] from 12.0 3 to  10 5.5
\arrow <5pt>  [0.2,0.4] from 10 5.5 to  5.9 4.4
\endpicture
\caption{\label{fig:ssb-bothern}{\small 
The simplices $\Sigma_A$ and $\Sigma_B$.   The closed curve $\Gamma$
travels in both triangles clockwise along the three legs. $\Gamma$ reverses at the points $f^A_{ij}$ and
$f^B_{ij}$.}}
\end{figure}

In this section we will make some precise calculations for
the periodic orbit $\tilde \Gamma$ on the Jitter set for the induced flow
on $\partial \Sigma$.
More precisely, we will consider the first return map to 
some first return section at some point in $\Gamma$.
To do this, consider the line $V_0$ in $\Sigma_A$ where 
player $B$ is indifferent between strategies $2$ and $3$
(it goes from $R^A_{23}$ to $Q^A_{23}$, see Figure~\ref{fig:ssb-bothern}
for the location of these points).
Similarly, let $V_1$ in $\Sigma_B$ be where player $A$ 
 is indifferent between strategies $2$ and $3$
(it goes from $R^B_{23}$ to $Q^B_{23}$)
and let $V_2$ be the line in $\Sigma_A$ where 
player $B$ is indifferent between strategies $1$ and $3$
(it goes from $R^A_{31}$ to $Q^A_{31}$).
Moreover, let $\partial \Sigma_A^{ij}$ be the side of $\partial \Sigma_A$
containing $R^A_{ij}$ and similarly define $\partial \Sigma_B^{ij}$ as the side
of $\partial \Sigma_B$ containing $R^B_{ij}$ .

Let $R_0$ be the first entry map $V_0\times \partial \Sigma_B^{31} 
\to \partial \Sigma_A^{12} \times V_1$
and $R_1$ be the first entry map $\partial \Sigma_A^{12}
\times V_1\to V_2\times \partial \Sigma_B^{12}$
corresponding to the induced flows on $\partial \Sigma$.
By the symmetry of the system,
the first return map  to $V_0\times \partial \Sigma_B^{31}$  is equal to the 
third iterate of $R_1\circ R_0$. (Provided we make sure we choose the axis
consistently.)

The first leg of the orbit $\Gamma$ (the one
which is contained in $Z_{1,2}^B\times Z_{1,3}^B\subset \Sigma_A\times \Sigma_B$
and which  is the first piece of the two legs
of the orbit $\Gamma$ shown in Figure~\ref{fig:ssb-bothern}) corresponds
to the first entry map $R_0$.
During the transition which corresponds to $R_0$, 
player $A$ only chooses between strategy $1$ and $3$ and player $B$ only chooses between strategy $2$ and $1$.  Note that as soon as the orbit hits $\Sigma_A\times 
V_1$ and until it hits $V_2\times \Sigma_B$,
player $A$ will only choose between strategies $1$ and $2$ (while player $B$
still only plays $1$ and $2$).
So the first entry maps to these sections allow us to consider
the pieces of the orbit where each players only switch between two strategies.

Let us describe how much further or closer an orbits near $\tilde \Gamma$ can get 
to $\tilde \Gamma$ while it orbits nearby.
To do this,  define the sum-distance (i.e. the metric $\dist$) on 
$\partial \Sigma\subset  \Sigma \subset \rz^6$
between two points $z=(z_1,\dots,z_6)$,
$w=(z_1,\dots,w_6)$ by 
$$\dist(z,w)=\sum_{i=1,\dots,6}\,\,  |z_i-w_i|.$$
This metric is well-suited to dealing with quadrilaterals.
Because of the discussion in the previous subsection,
there exist quadrilaterals in  $V_0\times \partial \Sigma_B^{13}$
which are mapped by $R_1$ into another quadrilaterals in 
$\partial \Sigma_A^{12} \times V_1$ (and similarly for $R_2$). 
Let us compute these quadrilaterals (up to first order).
It will be important to be consistent in the choice 
so let us write
$$V_0\times \partial \Sigma_B^{13}=[R^A_{23},Q_{23}^A]\times [P_1^B,P_3^B],$$
$$\partial \Sigma_A^{12} \times V_1=[P_1^A,P_3^A]\times [Q^B_{23},R^B_{23}],$$
$$V_2\times \partial \Sigma_B^{21}=[R^A_{31},Q^A_{31}]\times [P_2^B,P_1^B]$$
and number the half-lines clockwise starting with the positive horizontal axis,
see Figure~\ref{fig:axis}.

\begin{figure}[htp]
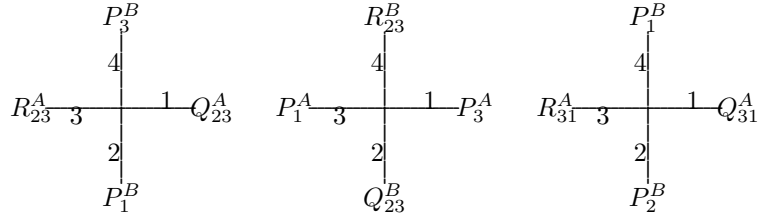
 \hfil
\beginpicture
\dimen0=0.1cm
\setcoordinatesystem units <\dimen0,\dimen0>
\setplotarea x from 0 to 10, y from -10 to 10
\plot -10 0 10 0 /
\plot -0 -10 0 10 /
\put {$R^A_{23}$} at -12 0
\put {$Q^A_{23}$} at 12 0
\put {$P_1^B$} at 0 -12
\put {$P_3^B$} at 0 12
\put {$1$} at 6 1
\put {$4$} at -1 6
\put {$3$} at -6 -1 
\put {$2$} at -1 -6 
\setcoordinatesystem units <\dimen0,\dimen0> point at -35 0
\setplotarea x from 0 to 10, y from -10 to 10  
\plot -10 0 10 0 /
\plot -0 -10 0 10 /
\put {$P_1^A$} at -12 0
\put {$P_3^A$} at 12 0
\put {$Q^B_{23}$} at 0 -12
\put {$R^B_{23}$} at 0 12
\put {$1$} at 6 1
\put {$4$} at -1 6
\put {$3$} at -6 -1 
\put {$2$} at -1 -6 
\setcoordinatesystem units <\dimen0,\dimen0> point at -70 0 
\setplotarea x from 0 to 10, y from -10 to 10
\plot -10 0 10 0 /
\plot -0 -10 0 10 /
\put {$1$} at 6 1
\put {$4$} at -1 6
\put {$3$} at -6 -1 
\put {$2$} at -1 -6 
\put {$R^A_{31}$} at -12 0
\put {$Q^A_{31}$} at 12 0
\put {$P_2^B$} at 0 -12
\put {$P_1^B$} at 0 12
\endpicture
\caption{\label{fig:axis}{\small The first return map to 
$[R^A_{23},Q_{23}^A]\times [P_1^B,P_3^B]$ is equal to the third iterate of
$R_1\circ R_0$ provided we identify $[R^A_{31},Q^A_{31}]\times [P_2^B,P_1^B]$ with $[R^A_{23},Q_{23}^A]\times [P_1^B,P_3^B]$  in a consistent way. For this reason we order  
the half-lines in the coordinate axis in the return sections as shown.
The origins of these axes are $(E^A,R^B_{31})$, $(R^A_{12},E^B)$ and $(E^A,R^B_{12})$.}}
\end{figure}

%(Here we take as usual
%$O(\epsilon^2)$ some function so that $O(\epsilon^2)/\epsilon^2$
%remains bounded as $\epsilon$ goes to zero.

\begin{prop}\label{prop:quadr}
For each $\epsilon>0$ small, there exists a quadrilateral  $\Q_{V_0,R_0}^\epsilon
\subset V_0\times \partial \Sigma_B^{13}$ with corners in the
coordinate axes, and such that the sum-distance of 
corners (1),(2),(3),(4) (labeled as in Figure~\ref{fig:axis} on the left)
to $(E^A,R^B_{13})$ are equal,  up to terms of order $\epsilon^2$, to 
\begin{equation}
\frac{2}{3}\frac{(2+\beta)\epsilon}{1+\beta+\beta^2}, \,\, 
 \frac{2\epsilon}{(2-\beta)(1+\beta)}, \,\, 
\frac{2}{3}\frac{(2+\beta)\epsilon}{\beta(1+\beta+\beta^2)}, \,\, 
\frac{2\epsilon}{2-\beta} .
\label{eq:quadrV0R0}
\end{equation}
Similarly there exists a quadrilateral 
$\Q_{V_1,R_0}^\epsilon \subset  \partial \Sigma_A^{12} \times V_1$
with corners in the
coordinate axes, and such that the sum-distance of these
corners (1),(2),(3),(4) (labeled as in Figure~\ref{fig:axis} in the middle)
to $(R^A_{12},E^B)$ are equal, up to terms of order $\epsilon^2$, to
\begin{equation}\begin{array}{rl}
& 
\dfrac{2(2-\beta)\epsilon}{(1+\beta)(2+\beta)}, \\ 
& \\
&\dfrac{2}{3}\dfrac{(4-4\beta+\beta^2)\epsilon}{1+\beta^3+\beta+\beta^4}
\mbox{ for }\beta\in (0,1/2) \mbox{ and }\dfrac{2}{3}\dfrac{(2-\beta)\epsilon}{1+\beta^3} \mbox{ for }\beta\in (1/2,1),\\
& \\
&\dfrac{2(2-\beta)\epsilon}{\beta(1+\beta)(2+\beta)}, \\
&\\
&
\dfrac{2}{3}\dfrac{(4-4\beta+\beta^2)\epsilon}{(1+\beta^3)} \mbox{ for }\beta\in (0,1/2)
\mbox{ and } \dfrac{2}{3}\dfrac{(2-\beta)\epsilon}{1-\beta+\beta^2} \mbox{ for }\beta\in (1/2,1).
\end{array}\label{eq:quadrV1R0}
\end{equation}
The first entry map $R_0$ maps 
$Q_{V_0,R_0}^\epsilon$ into $Q_{V_1,R_0}^\epsilon$.
\end{prop}

\begin{figure}[htp]
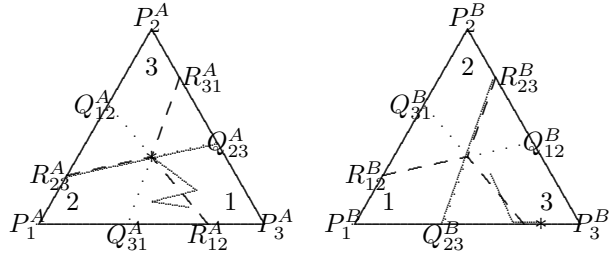
 \hfil
\beginpicture
\dimen0=0.15cm
\setcoordinatesystem units <\dimen0,\dimen0>
\setplotarea x from 0 to 30, y from -5 to 20
%\put {$\beta>0$} at 22 9
\setlinear
\setsolid
\plot 0 0 20 0 /
\plot 0 0 10 17.3 20 0 /
\setdashes
\plot 15 0 10 5.98 /
\plot 2.5 4.33 10 5.98 /
\setdots <0.3mm>
\plot 2.5 4.2 15.80 7.13 /
\setdashes
\plot 12.5 13 10 5.98 /
\setdots
\plot 10 5.98 6.14 10.60 /
\plot 10 5.98 15.80 7.26 /
\plot 10 5.98 7.88 0 /
\put {$*$} at 10. 5.98 
\setdots <0.3mm> 
\plot 10 5.98 14 3 10 2 13.3 1.5 / %4 2 /
\setsolid
%\multiput {*} at 11.5 9.8 /  %11.4 9.8 /
%\multiput {*} at 13.1 2.0 5.9 4.8 /
%\multiput {$\bullet$} at 15 0 2.5 4.33 12.5 13 /
%\put {$f_{12}^A$} at 14.5 3
%\put {$f_{31}^A$} at 10 9.8
%\put {$f_{23}^A$} at 5.9 6.5
%\put {$\alpha^A$} at 13 1       % added
%\put {$\alpha^A+\pi/3$} at 17.4 1 % added
%\put {$\frac{2\pi}{3}-\alpha^A$} at 5 9 % added
\put {$R_{12}^A$} at 15 -1
\put {$R_{31}^A$} at 14.5 13
\put {$R_{23}^A$} at 1 4.33
\put {$Q^A_{12}$} at 5 10.60
\put {$Q^A_{23}$} at 16.80 7.26
\put {$Q^A_{31}$} at 7.88 -1
\put {$P^A_1$} at -1 0
\put {$P^A_3$} at 21 0
\put {$P^A_2$} at 10 18.3
\put {$2$} at  3 2
\put {$3$} at 10 14
\put {$1$} at 17 2
%\arrow <5pt>  [0.2,0.4] from 11 5.5 to 13 3
%\arrow <5pt>  [0.2,0.4] from 12.3 2.3 to 10.2 5.0
%\circulararc -180 degrees from 13 3 center at 12.7  2.6
\setcoordinatesystem units <\dimen0,\dimen0> point at -28 0
\setplotarea x from 0 to 30, y from -5 to 20
\setsolid
\plot 0 0 20 0 /
\plot 0 0 10 17.3 20 0 /
\setdashes
\plot 15 0 10 5.98 /
\plot 2.5 4.33 10 5.98 /
\plot 12.5 13 10 5.98 /
\setdots <0.3mm>
\plot 7.68 0 12.3 13 /
\setdots
\plot 10 5.98 6.14 10.60 /
\plot 10 5.98 15.80 7.26 /
\plot 10 5.98 7.88 0 /
\setsolid
%\multiput {*} at 11.4 9.8 12.4 3 5.9 4.9 /
%\multiput {$\bullet$} at 15 0 2.5 4.33 12.5 13 /
%\put {$f_{31}^B$} at 14.0 3
%\put {$f_{23}^B$} at 10 9.8
%\put {$f_{12}^B$} at 5.9 6.5
%\put {$R_{31}^B$} at 15 -1
%\put {$\alpha^B$} at 13 1       % added
%\put {$\alpha^B+\pi/3$} at 17.4 1 % added
%\put {$\frac{2\pi}{3}-\alpha^B$} at 5 9 % added
\put {$R_{23}^B$} at 14.5 13
\put {$R_{12}^B$} at 1 4.33
\put {$Q^B_{31}$} at 5 10.60
\put {$Q^B_{12}$} at 16.80 7.26 %%% WAS WRONG SSSSS
\put {$Q^B_{23}$} at 7.88 -1 %%% WAS WRONG SSSSS
\put {$P^B_1$} at -1 0
\put {$P^B_3$} at 21 0
\put {$P^B_2$} at 10 18.3
\put {$1$} at  3 2
\put {$2$} at 10 14
\put {$3$} at 17 2
%\arrow <5pt>  [0.2,0.4] from 12.0 3 to  10 5.5
%\arrow <5pt>  [0.2,0.4] from 10 5.5 to  5.9 4.4
\put {$*$} at 16.5 0 
\setdots <0.3mm>
\plot 16.5 0.2 14 0.2 12 4.6 / 
\endpicture
\caption{\label{fig:ssb-botherna}{\small 
The simplices $\Sigma_A$ and $\Sigma_B$.   The line $V_0\subset \Sigma^A$ is drawn in the left triangle
and the line $V_1\subset \Sigma^B$ on the right, and the first 
four line segments of the orbit (starting at time $t_0=0$ at the point $*$) which are
computed in the proof of Proposition~\ref{prop:quadr} are shown.}}
\end{figure}

So this proposition gives precise information on how much further
or closer one gets to $\tilde \Gamma$
during the transition from $V_0\times \partial \Sigma_B^{13}$
to $ \partial \Sigma_A^{12} \times V_1$. Remember that
we saw in the previous subsection that
the angle of $R_0(z)$ depends extremely sensitively on $\epsilon$
and so it essentially suffices to compare the size of the terms
in (\ref{eq:quadrV0R0}) to those in  (\ref{eq:quadrV0R0}).
%Indeed, note that the last two terms in (\ref{eq:quadrV0R0}) are larger than the first two,
%the third is decreasing in $\beta$ and the fourth one increasing in $\beta$.
%The third and fourth term in (\ref{eq:quadrV0R0}) are equal for
%some $\beta\in [0.6,0.62]$. Similarly,  the last two terms in (\ref{eq:quadrV1R0}) 
%are larger than the first two, and they are all decreasing.
%In this way we can easily prove that (for $\epsilon$ sufficiently small)
%the maximal expansion when going from
%$\Q_{V_0,R_0}^\epsilon$ to  $\Q_{V_1,R_0}^\epsilon$
%is at least $0.5$:
%\begin{equation}
%\dfrac
%{\max\{ \mbox{term in (\ref{eq:quadrV1R0})}\}}
%{\min\{ \mbox{term in (\ref{eq:quadrV0R0})}\}} \ge 0.5 \mbox{ for all }\beta\in [0,1]
%\end{equation}
%(the expansion is much larger when $\beta$ is close to zero)
%whereas the maximal amount of contraction when going from
%$\Q_{V_0,R_0}^\epsilon$ to  $\Q_{V_1,R_0}^\epsilon$
%is at most $0.6$:
%\begin{equation}
%\dfrac
%{\min\{ \mbox{term in \ref{eq:quadrV1R0})}\}}
%{\max\{ \mbox{term in \ref{eq:quadrV0R0})}\}} \le 0.6  \mbox{ for all }\beta\in [0,1] .
%\label{EXCONV0V1}
%\end{equation}
%Although this information seems meaningless because a 
%point in $\Q_{V_0,R_0}^\epsilon$ gets to a specific point in $\Q_{V_1,R_0}^\epsilon$,
%we should remember that the angular part is highly dependent on 
%$\epsilon$; so in some sense a point $\Q_{V_0,R_0}^\epsilon$
%can reach 'any point' in $\Q_{V_1,R_0}^\epsilon$. We shall make this precise
%in the next subsection and exploit the above two estimates.

%\medskip

What we need to do in the proof of Proposition~\ref{prop:quadr} is
to associate to $R_0$ invariant cones as in the previous subsection,
and the corresponding quadrangles in $V_0\times \partial \Sigma_B^{13}$
 and in $ \partial \Sigma_A^{12} \times V_1$.
After that, we will do the same for the first entry map $R_1$.

\begin{proof}
Since we want to consider the induced flow on $\partial \Sigma$,
we take a starting point 
$(p^A,p^B)\in V_0\times \partial \Sigma_B^{12}$ when considering the map $R_0$.
During this part of the orbit, the orbit jitters around this first leg of $\Gamma$.
To describe this precisely, we explicitly compute the quadrilateral from the previous subsection.
As we have seen in the previous subsection the orbit
is contained in a cone with apex the 'cone-target'
which in this case is equal to $(T_A,T_B)=(R^A_{12},Q_{13}^B)$.
To compute this cone, let us take as a special starting point in
$V_0\times \partial \Sigma_B^{13}$
the point 
$p_A=(1/3,1/3,1/3), p_B=(1-\beta-\epsilon, 0, 1+\epsilon)/(2-\beta)$
(where $p_B=R^B_{13}$ when $\epsilon=0$) and
 compute the first four pieces where this orbit
aims for $(P_3^A,P_1^B)$, $(P_1^A,P_1^B)$, $(P_1^A,P_2^B)$ and $(P^A_3,P^B_2)$
under the original flow (i.e. the first four times when the players
hit an indifference plane under the original flow). Since the calculations are rather laborious
and it is easy to make a mistake, we did this by using Maple (the worksheet 
can be requested from the authors - and also is available on the first author's webpage).
For simplicity we take the parametrisation 
$p^A(t_k+s)=p^A(t_k)(1-s)+sP^A_i$ and $p^A(t_k+s)=p^A(t_k)(1-s)+sP^A_j$ 
for all $t\in (t_k,t_{k+1})$ provided $A$ resp. $B$ aim for $P^A_i,P^B_j$
during this time interval.
The first hitting time is at $\hat t_1:=t_1 := \epsilon/(1+\epsilon)$, 
and then
$$p_A(\hat t_1)= \left( \dfrac{1}{3(1+\epsilon)},\dfrac{1}{3(1+\epsilon)}, \dfrac{1+3\epsilon}{3(1+\epsilon)}\right) \mbox{ ,  }\quad 
p_B(\hat t_1)=\left( \frac {1-\beta}{2-\beta},  0, \frac{1}{2-\beta} \right).$$
The next time the players hit an indifference plane is at $\hat t_2:=t_1+t_2$
where $t_2:= \epsilon/(\beta+\beta\epsilon+1+2\epsilon)$
and then 
$$p_A(\hat t_2):=\left(  
\dfrac{1}{3}\dfrac{\epsilon+\beta+1}{\beta+\beta\,\epsilon+1+2\,\epsilon} \,\,
, \,\,
\dfrac{1}{3} \dfrac {\beta+1}{\beta+\beta\,\epsilon+1+2\,\epsilon}\, \, , \, \, \dfrac{1}{3}
\dfrac { \left( 1+3\,\epsilon \right)  \left( \beta+1 \right) }{\beta+\beta\,\epsilon+1+2\,\epsilon}\right),
$$
$$p_B(\hat t_2):=\left( \frac {{\beta}^{2}+{\beta}^{2}\epsilon-1-3\,\epsilon+\beta\,\epsilon}{ \left( \beta+\beta\,\epsilon+1+2\,\epsilon \right)  \left( -2+\beta \right) }
\,\, , \,\, 0 \,\, , \,\, -\frac {\beta+\beta\,\epsilon+1+\epsilon}{ \left( \beta+\beta\,\epsilon+1+2\,\epsilon \right)  \left( -2+\beta \right) }\right)$$
Then it hits at $\hat t_3:=t_1+t_2+t_3$ with $t_3:=\epsilon/({\beta}^{2}+{\beta}^{2}\epsilon+\beta+2\,\beta\,\epsilon+\epsilon)$
and then $p_A(\hat t_3)$ and $p_B(\hat t_3)$ are equal to 
$$\left(\dfrac{1}{3}\frac {\beta+3\,\epsilon}{\beta+\beta\,\epsilon+\epsilon}
\,\, , \,\, \frac{1}{3}\dfrac {\beta}{\beta+\beta\,\epsilon+\epsilon} \,\, , \,\, \dfrac{1}{3} 
\frac {\beta\, \left( 1+3\,\epsilon \right) }{\beta+\beta\,\epsilon+\epsilon}\right) \,\, , 
$$
$$\left(\frac { \left( {\beta}^{2}+{\beta}^{2}\epsilon-1-3\,\epsilon+\beta\,\epsilon \right) \beta}
{ \left( -2+\beta \right)  \left( {\beta}^{2}+{\beta}^{2}\epsilon+\beta+2\,\beta\,\epsilon+\epsilon \right) }
\,\, , \,\, 
\frac {\epsilon}{{\beta}^{2}+{\beta}^{2}\epsilon+\beta+2\,\beta\,\epsilon+\epsilon}
\,\, , \,\, 
-\frac { \left( 1+\epsilon \right) \beta}{ \left( -2+\beta \right)  \left( \beta+\beta\,\epsilon+\epsilon \right) } \right)
$$
and again at $\hat t_4:=\hat t_3+t_4$ where $t_4:=\epsilon/(\beta+\beta\epsilon+2\epsilon)$ and then
$p_A(\hat t_4)$ and $p_B(\hat t_4)$ are 
$$\left(\frac{1}{3} \frac {\beta+3\,\epsilon}{\beta+\beta\,\epsilon+2\,\epsilon}
\,\, , \,\, 
\frac{1}{3} \frac {\beta}{\beta+\beta\,\epsilon+2\,\epsilon}
\,\, , \,\, 
\frac{1}{3} \frac {\beta+3\,\beta\,\epsilon+3\,\epsilon}{\beta+\beta\,\epsilon+2\,\epsilon} \right),$$
$$
\left( \frac {\beta\, \left( {\beta}^{2}+{\beta}^{2}\epsilon-1-3\,\epsilon+\beta\,\epsilon \right) }{ \left( \beta+1 \right)  \left( -2+\beta \right)  \left( \beta+\beta\,\epsilon+2\,\epsilon \right) }
\,\, , \,\, 
\frac { \left( \beta+2 \right) \epsilon}{ \left( \beta+1 \right)  \left( \beta+\beta\,\epsilon+2\,\epsilon \right) } 
\,\, , \,\,  -\frac {\beta\, \left( 1+\epsilon \right) }{ \left( -2+\beta \right)  \left( \beta+\beta\,\epsilon+2\,\epsilon \right) }\right).$$
Next we compute the cone. As mentioned, the cone-targets are 
$$T_A:=\left( \frac{1}{\beta+2}\,\, , \,\, 0 \,\, , \,\, \frac {\beta+1}{\beta+2} \right)\mbox{ and }T_B:=\left(  \frac {\beta}{\beta+1}\,\, , \,\,  \frac{1}{\beta+1}\,\, , \,\, 0 \right)$$
and we compute the intersection of the line through $(T_A,T_B)$ and
the points $(p_A(\hat t_i),p_B(\hat t_i))$, $i=1,\dots,4$ 
with the three-dimensional section $V_0\times \Sigma_B$.
This gives an intersection point  at $\hat t_1$,
$$\left( 1/3\frac {-3\,\beta\,\epsilon+{\beta}^{2}+\beta+1}{{\beta}^{2}+\beta-\beta\,\epsilon+1+\epsilon}
\,\, , \,\, 
1/3
\frac {{\beta}^{2}+\beta+1}{{\beta}^{2}+\beta-\beta\,\epsilon+1+\epsilon}
\,\, , \,\, 
1/3
\frac {{\beta}^{2}+\beta+1+3\,\epsilon}{{\beta}^{2}+\beta-\beta\,\epsilon+1+\epsilon}\right),$$
$$
\left(\frac{ {4\beta}^{2}\epsilon+{\beta}^{4}+{\beta}^{3}+{\beta}^{3}\epsilon-\beta-1-\beta\,\epsilon-\epsilon}
{ \left( {\beta}^{2}+\beta-\beta\,\epsilon+1+\epsilon \right)  \left( \beta+1 \right)  \left( -2+\beta \right) }
\,\, , \,\, 
-\frac {\beta\,\epsilon\, \left( \beta+2 \right) }{ \left( {\beta}^{2}+\beta-\beta\,\epsilon+1+\epsilon \right)  \left( \beta+1 \right) }
\,\, , \,\, \right. $$
$$\left. -\frac { \left( {\beta}^{2}+\beta+1 \right)  \left( 1+\epsilon \right) }{ \left( {\beta}^{2}+\beta-\beta\,\epsilon+1+\epsilon \right)  \left( -2+\beta \right) } \right);
$$
at $\hat t_2$ the intersection point is
$$(1/3,1/3,1/3),$$
$$\left(
\frac {2\,\beta\,\epsilon+{\beta}^{3}+{\beta}^{2}+2\,{\beta}^{2}\epsilon-\beta-1-3\,\epsilon}{ \left( \beta+1 \right) ^{2} \left( -2+\beta \right) }
\,\, , \,\, 
-{\frac {\epsilon\, \left( \beta+2 \right) }{ \left( \beta+1 \right) ^{2}}}
\,\, , \,\, 
-\frac {1+\epsilon}{-2+\beta}\right); $$
at $\hat t_3$ the intersection point is
$$ \left(1/3\frac { \left( {\beta}^{2}+\beta+1+3\,\epsilon \right) \beta}{{\beta}^{2}+\beta+{\beta}^{3}+\beta\,\epsilon-\epsilon}
\,\, , \,\, 
1/3\frac { \left( {\beta}^{2}+\beta+1 \right) \beta}{{\beta}^{2}+\beta+{\beta}^{3}+\beta\,\epsilon-\epsilon}
\,\, , \,\, 
1/3
\frac {{\beta}^{3}+{\beta}^{2}+\beta-3\,\epsilon}{{\beta}^{2}+\beta+{\beta}^{3}+\beta\,\epsilon-\epsilon}\right)
$$
$$
\left( \frac {\beta\, \left( {\beta}^{4}+2\,{\beta}^{3}\epsilon+{\beta}^{3}+2\,{\beta}^{2}\epsilon-2\,\beta\,\epsilon-\beta+\epsilon-1 \right) }{ \left( \beta+1 \right)  \left( -2+\beta \right)  \left( {\beta}^{2}+\beta+{\beta}^{3}+\beta\,\epsilon-\epsilon \right) }
\,\, , \,\, 
-\frac {\epsilon\, \left( {\beta}^{3}+{\beta}^{2}+1 \right) }{ \left( {\beta}^{2}+\beta+{\beta}^{3}+\beta\,\epsilon-\epsilon \right)  \left( \beta+1 \right) }
\,\, ,\,\, \right.
$$
$$\left. -\frac { \left( {\beta}^{2}+\beta+1 \right)  \left( 1+\epsilon \right) \beta}{ \left( {\beta}^{2}+\beta+{\beta}^{3}+\beta\,\epsilon-\epsilon \right)  \left( -2+\beta \right)} \right)
$$
while at $\hat t_4$ the intersection point we get is the original starting point
(this is not surprising because the orbit lies on the cone through these points):
$$\left( 1/3 , 1/3 , 1/3\right) \mbox{ , } \left( \frac {1-\beta}{2-\beta},  0, \frac{1}{2-\beta} \right).$$
These four points in $\Sigma_A\times \Sigma_B$ together with
the apex $(T_A,T_B)$ determine a cone (for each $\epsilon>0)$.

However, remember we want to compute the cone for the induced flow.
This means that we have to take the intersection of
the lines from $(E^A,E^B)$ through these points with $\partial \Sigma$.
Since $\epsilon>0$ is small, these points will be contained in $V_0\times \partial \Sigma_B^{13}$.
This gives at $\hat t_1$,
$$\left( 1/3
\frac {{\beta}^{3}+2\,{\beta}^{2}+2\,\beta+1+3\,\beta\,\epsilon}{{\beta}^{3}+2\,{\beta}^{2}+2\,\beta+2\,{\beta}^{2}\epsilon+1+\epsilon+6\,\beta\,\epsilon}
\,\, , \,\,
1/3{\frac {{\beta}^{3}+2\,{\beta}^{2}+2\,\beta+3\,{\beta}^{2}\epsilon+1+6\,\beta\,\epsilon}{{\beta}^{3}+2\,{\beta}^{2}+2\,\beta+2\,{\beta}^{2}\epsilon+1+\epsilon+6\,\beta\,\epsilon}}
\,\, , \,\, \right. $$
$$\left. 1/3\frac {{\beta}^{3}+2\,{\beta}^{2}+2\,\beta+3\,{\beta}^{2}\epsilon+1+3\,\epsilon+9\,\beta\,\epsilon}{{\beta}^{3}+2\,{\beta}^{2}+2\,\beta+2\,{\beta}^{2}\epsilon+1+\epsilon+6\,\beta\,\epsilon} \right)
$$
$$\left( \frac {-1+\beta}{-2+\beta}\,\, , \,\, 0 \,\, , \,\, 
\frac{1}{2-\beta} \right)$$
at $\hat t_2$,
$$\left( 1/3 , 1/3 , 1/3 \right),$$
$$\left( \frac {{\beta}^{3}-\beta-1+3\,{\beta}^{2}\epsilon-7\,\epsilon+2\,\beta\,\epsilon+{\beta}^{2}}{ \left( {\beta}^{2}+2\,\beta+1+3\,\beta\,\epsilon+6\,\epsilon \right)  \left( -2+\beta \right) }
\,\, , \,\, 0
\,\, , \,\, 
-\frac {2\,\beta+1+5\,\epsilon+{\beta}^{2}+2\,\beta\,\epsilon}{ \left( {\beta}^{2}+2\,\beta+1+3\,\beta\,\epsilon+6\,\epsilon \right)  \left( -2+\beta \right) }\right) $$
at
$\hat t_3$,
$$\left( 1/3\frac {2\,{\beta}^{3}+2\,{\beta}^{2}+\beta+{\beta}^{4}+6\,{\beta}^{2}\epsilon+3\,\epsilon+3\,{\beta}^{3}\epsilon+3\,\beta\,\epsilon}{2\,{\beta}^{3}+2\,{\beta}^{2}+\beta+{\beta}^{4}+4\,{\beta}^{2}\epsilon+2\,\epsilon+3\,{\beta}^{3}\epsilon}
\,\, , \,\, 
1/3\frac {2\,{\beta}^{3}+2\,{\beta}^{2}+\beta+{\beta}^{4}+3\,{\beta}^{2}\epsilon+3\,\epsilon+3\,{\beta}^{3}\epsilon}{2\,{\beta}^{3}+2\,{\beta}^{2}+\beta+{\beta}^{4}+4\,{\beta}^{2}\epsilon+2\,\epsilon+3\,{\beta}^{3}\epsilon}
\,\, , \,\,  \right. $$
$$\left. 1/3 \frac {\beta\, \left( 2\,{\beta}^{2}+2\,\beta+1+{\beta}^{3}+3\,\beta\,\epsilon+3\,{\beta}^{2}\epsilon-3\,\epsilon \right) }{2\,{\beta}^{3}+2\,{\beta}^{2}+\beta+{\beta}^{4}+4\,{\beta}^{2}\epsilon+2\,\epsilon+3\,{\beta}^{3}\epsilon} \right)
$$
$$\left( \frac {1-\beta}{2-\beta} \,\, , \,\,  0 \,\, , \,\,  \frac{1}{ 2- \beta} \right)$$
and at 
$\hat t_4$ again the point we started with
$$\left ( 1/3 , 1/3 , 1/3 \right) \,\, , \,\, 
\left( \frac {\beta+\epsilon-1}{-2+\beta} \,\, , \,\,  0 \,\, , \,\, \frac {1-\epsilon}{-2+\beta} \right) .$$
These four points determine a quadrangle $\Q^\epsilon_{V_0,R_0}$ 
in  $V_0\times \partial \Sigma_B^{13}$.

Next we find the intersection points of 
these cones with $\Sigma_A\times V_1$
and then take the intersections with $\partial \Sigma$ 
of half-lines from $E$ in the direction of these points.
Since the expressions are rather similar to the ones before,
we will only give these final points in $\partial \Sigma_B \times V_1$.
The intersection corresponding to  $\hat t_1$ is 
$$\left(
\frac {\beta+\beta\,\epsilon+1+\epsilon}{ \left( 3\,\epsilon+\beta+1 \right)  \left( \beta+2 \right) }\,\, ,  \,\, 0
\,\, , \,\, \frac {2\,\beta+1+5\,\epsilon+{\beta}^{2}+2\,\beta\,\epsilon}{ \left( 3\,\epsilon+\beta+1 \right)  \left( \beta+2 \right) }\right) \,\, , \,\, 
\left( 1/3 , 1/3 , 1/3 \right),$$
to $\hat t_2$ is 
$$\left( \frac{1}{\beta+2} \,\, , \,\, 0 \,\, , \,\, \frac {\beta+1}{\beta+2} \right),$$
$$\left( 1/3\frac {3\,{\beta}^{2}\epsilon-9\,\beta\,\epsilon+1+9\,\epsilon+{\beta}^{3}+3\,{\beta}^{3}\epsilon+\beta+{\beta}^{4}}{2\,{\beta}^{2}\epsilon-5\,\beta\,\epsilon+1+5\,\epsilon+{\beta}^{3}+3\,{\beta}^{3}\epsilon+\beta+{\beta}^{4}}
\,\, , \,\, 
1/3
\frac {3\,{\beta}^{2}\epsilon-6\,\beta\,\epsilon+1+3\,\epsilon+{\beta}^{3}+3\,{\beta}^{3}\epsilon+\beta+{\beta}^{4}}{2\,{\beta}^{2}\epsilon-5\,\beta\,\epsilon+1+5\,\epsilon+{\beta}^{3}+3\,{\beta}^{3}\epsilon+\beta+{\beta}^{4}}
\,\, , \,\,  \right. $$
$$\left. 1/3 \frac {1+{\beta}^{3}+\beta+{\beta}^{4}+3\,\epsilon+3\,{\beta}^{3}\epsilon}{2\,{\beta}^{2}\epsilon-5\,\beta\,\epsilon+1+5\,\epsilon+{\beta}^{3}+3\,{\beta}^{3}\epsilon+\beta+{\beta}^{4}}\right), $$
to $\hat t_3$, 
$$\left( \frac {2\,\beta\,\epsilon+{\beta}^{2}+\beta+2\,\epsilon}{ \left( 3\,\epsilon+\beta+1 \right) \beta\, \left( \beta+2 \right) }
\,\, , \,\, 0 \,\, , \,\, 
\frac {3\,{\beta}^{2}\epsilon+4\,\beta\,\epsilon+{\beta}^{3}+2\,{\beta}^{2}+\beta-2\,\epsilon}{ \left( 3\,\epsilon+\beta+1 \right) \beta\, \left( \beta+2 \right) }\right) \,\, , \,\, 
\left( 1/3 , 1/3 , 1/3 \right), $$
and to $\hat t_4$,
$$\left( \frac{1}{\beta+2}\,\, , \,\, 0 \,\, , \,\, \frac {\beta+1}{\beta+2} \right), $$
$$\left(  1/3 \frac {6\,\beta\,\epsilon+1-3\,\epsilon+{\beta}^{3}}{2\,\beta\,\epsilon+1+\epsilon+{\beta}^{2}\epsilon+{\beta}^{3}}
\,\, , \,\, 
1/3\frac {{\beta}^{2}-\beta+1+3\,\epsilon}{{\beta}^{2}-\beta+\beta\,\epsilon+1+\epsilon}
\,\, , \,\, 
1/3\frac {1+{\beta}^{3}-3\,\beta\,\epsilon+3\,\epsilon+3\,{\beta}^{2}\epsilon}{2\,\beta\,\epsilon+1+\epsilon+{\beta}^{2}\epsilon+{\beta}^{3}} \right) .
$$
These points form the quadrilaterals $\Q_{V_1,R_0}^\epsilon$ in 
$\partial \Sigma_B^{12}\times V_1$. 

To get Proposition~\ref{prop:quadr} we differentiate the points
forming the quadrilaterals $\Q_{V_0,R_0}^\epsilon$ and 
$\Q_{V_1,R_0}^\epsilon$ with respect to $\epsilon$. Since these points correspond
to probability vectors, the sum of these derivatives is equal to zero. 
So we merely need to take the sum of the absolute values of these
derivatives (or twice the positive terms).

\end{proof}

Since the calculations from $R_1$ are similar to those
done in the  Proposition~\ref{prop:quadr} we shall not show them here.
The maple worksheet in which these computations are done
can be obtained from the authors on request.

\begin{figure}[htp]
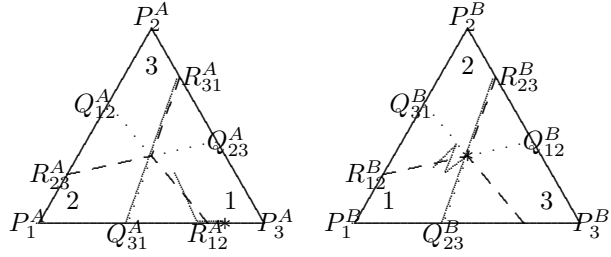
 \hfil
\beginpicture
\dimen0=0.15cm
\setcoordinatesystem units <\dimen0,\dimen0>
\setplotarea x from 0 to 30, y from -5 to 20
%\put {$\beta>0$} at 22 9
\setlinear
\setsolid
\plot 0 0 20 0 /
\plot 0 0 10 17.3 20 0 /
\setdashes
\plot 15 0 10 5.98 /
\plot 2.5 4.33 10 5.98 /
\setdots <0.3mm>
\plot 7.68 0 12.3 13 /
%\setdots <0.3mm>
%\plot 2.5 4.2 15.80 7.13 /
\setdashes
\plot 12.5 13 10 5.98 /
\setdots
\plot 10 5.98 6.14 10.60 /
\plot 10 5.98 15.80 7.26 /
\plot 10 5.98 7.88 0 /
%\put {$*$} at 10. 5.98 
%\setdots <0.3mm> 
%\plot 10 5.98 14 3 10 2 13.3 1.5 / %4 2 /
\setsolid
%\multiput {*} at 11.5 9.8 /  %11.4 9.8 /
%\multiput {*} at 13.1 2.0 5.9 4.8 /
%\multiput {$\bullet$} at 15 0 2.5 4.33 12.5 13 /
%\put {$f_{12}^A$} at 14.5 3
%\put {$f_{31}^A$} at 10 9.8
%\put {$f_{23}^A$} at 5.9 6.5
%\put {$\alpha^A$} at 13 1       % added
%\put {$\alpha^A+\pi/3$} at 17.4 1 % added
%\put {$\frac{2\pi}{3}-\alpha^A$} at 5 9 % added
\put {$R_{12}^A$} at 15 -1
\put {$R_{31}^A$} at 14.5 13
\put {$R_{23}^A$} at 1 4.33
\put {$Q^A_{12}$} at 5 10.60
\put {$Q^A_{23}$} at 16.80 7.26
\put {$Q^A_{31}$} at 7.88 -1
\put {$P^A_1$} at -1 0
\put {$P^A_3$} at 21 0
\put {$P^A_2$} at 10 18.3
\put {$2$} at  3 2
\put {$3$} at 10 14
\put {$1$} at 17 2
%\arrow <5pt>  [0.2,0.4] from 11 5.5 to 13 3
%\arrow <5pt>  [0.2,0.4] from 12.3 2.3 to 10.2 5.0
%\circulararc -180 degrees from 13 3 center at 12.7  2.6
\put {$*$} at 16.5 0 
\setdots <0.3mm>
\plot 16.5 0.2 14 0.2 12 4.6 / 
\setcoordinatesystem units <\dimen0,\dimen0> point at -28 0
\setplotarea x from 0 to 30, y from -5 to 20
\setsolid
\plot 0 0 20 0 /
\plot 0 0 10 17.3 20 0 /
\setdashes
\plot 15 0 10 5.98 /
\plot 2.5 4.33 10 5.98 /
\plot 12.5 13 10 5.98 /
\setdots <0.3mm>
\plot 7.68 0 12.3 13 /
\setdots
\plot 10 5.98 6.14 10.60 /
\plot 10 5.98 15.80 7.26 /
\plot 10 5.98 7.88 0 /
\setsolid
%\multiput {*} at 11.4 9.8 12.4 3 5.9 4.9 /
%\multiput {$\bullet$} at 15 0 2.5 4.33 12.5 13 /
%\put {$f_{31}^B$} at 14.0 3
%\put {$f_{23}^B$} at 10 9.8
%\put {$f_{12}^B$} at 5.9 6.5
%\put {$R_{31}^B$} at 15 -1
%\put {$\alpha^B$} at 13 1       % added
%\put {$\alpha^B+\pi/3$} at 17.4 1 % added
%\put {$\frac{2\pi}{3}-\alpha^B$} at 5 9 % added
\put {$R_{23}^B$} at 14.5 13
\put {$R_{12}^B$} at 1 4.33
\put {$Q^B_{31}$} at 5 10.60
\put {$Q^B_{12}$} at 16.80 7.26 %%% WAS WRONG SSSSS
\put {$Q^B_{23}$} at 7.88 -1 %%% WAS WRONG SSSSS
\put {$P^B_1$} at -1 0
\put {$P^B_3$} at 21 0
\put {$P^B_2$} at 10 18.3
\put {$1$} at  3 2
\put {$2$} at 10 14
\put {$3$} at 17 2
%\arrow <5pt>  [0.2,0.4] from 12.0 3 to  10 5.5
%\arrow <5pt>  [0.2,0.4] from 10 5.5 to  5.9 4.4
\put {$*$} at 10 5.98 
\setdots <0.3mm>
\plot 10 5.98 8 4.5 9 7 7 5  / 
\endpicture
\caption{\label{fig:ssb-bothernb}{\small 
The simplices $\Sigma_A$ and $\Sigma_B$.   The line $V_1\subset \Sigma^B$ is drawn in the right triangle
and the line $V_2\subset \Sigma^A$ on the left, and the first 
four line segments of the orbit (starting at time $t_0=0$ at the point $*$) which are
computed in the proof of Proposition~\ref{prop:quadr} are shown. Notice that $t_1,t_2,t_3,t_4$ correspond to half-lines 
(2), (3), (4), (1) in the middle of Figure~\ref{fig:axis} and to (4), (1), (2), (3) on the right. This can be seen by taking
the cones through the target points, intersect them respectively with $\Sigma^A\times V_1$ and
$V_2\times \Sigma^B$, and then  project from $E$ on the boundary of $\Sigma$.}}
\end{figure}

\begin{prop}\label{prop:quadr1}
For each $\epsilon>0$ small, there exists a quadrilateral  
$\Q_{V_1,R_1}^\epsilon \subset  \partial \Sigma_A^{12} \times V_1$
 with corners in the
coordinate axes, and such that the sum-distance of these
corners  (1),(2),(3),(4) (labeled as in Figure~\ref{fig:axis})
to $(R^B_{12},E^B)$ 
are equal,  up to terms of order $\epsilon^2$, to 
\begin{equation} \begin{array}{rl}
& \dfrac{2\epsilon}{2+\beta} \\
& \\
& \dfrac{2}{3}\dfrac{(2-3\beta+\beta^2)\epsilon}{1+\beta^3}\mbox{ for }\beta\in (0,1/2)\mbox{ and }
 \dfrac{2}{3}\dfrac{1-\beta}{1-\beta+\beta^2} \mbox{ for }\beta\in (1/2,1), \\
& \\ 
&  \dfrac{2(1-\beta)\epsilon}{2+\beta}, \quad \\
& \\
&\dfrac{2}{3}\dfrac{(2-3\beta+\beta^2)\epsilon}{\beta(1-\beta+\beta^2)} \mbox{ for }\beta\in (0,1/2)
\mbox{ and } \dfrac{2}{3}\dfrac{(1-\beta^2)\epsilon}{\beta(1-\beta+\beta^2)} \mbox{ for }\beta\in (1/2,1)
\end{array}\label{eq:quadrV1R1}\end{equation}
Similarly there exists a quadrilateral 
$\Q_{V_2,R_1}^\epsilon \subset V_2\times \partial \Sigma_B^{12}$
with corners in the
coordinate axes, and such that the sum-distance of these
corners  (1),(2),(3),(4) (labeled as in the figure on the right in Figure~\ref{fig:axis})
corners to $(E^A,R^A_{12})$ are equal, up to terms of order $\epsilon^2$, to
\begin{equation}\begin{array}{rl}
& \dfrac{2}{3}\dfrac{(4-3\beta^2-\beta^3)\epsilon}{1+3\beta+3\beta^2+2\beta^3},\,\,
 \dfrac{2(2-\beta-\beta^2)\epsilon}{(1+2\beta)(2-\beta)\beta}, \\
& \\
& \quad  \dfrac{2}{3}\dfrac{(11+21\beta+15\beta^2+7\beta^3)\epsilon}{(1+3\beta+3\beta^2+2\beta^3)(2+\beta)} ,\,\, 
\dfrac{2(2-\beta-\beta^2)\epsilon}{(1+2\beta)(1+\beta)(2-\beta)}.
\end{array}
\label{eq:quadrV2R1}
\end{equation}
The first entry map $R_1$ maps 
$Q_{V_1,R_1}^\epsilon$ into $Q_{V_2,R_1}^\epsilon$.
\end{prop}

%Similarly as before we compare the size of the terms
%appearing in ( \ref{eq:quadrV1R1}) and ( \ref{eq:quadrV2R1}).

%\begin{equation}
%\dfrac
%{\max\{ \mbox{term in (\ref{eq:quadrV2R1})}\}}
%{\min\{ \mbox{term in (\ref{eq:quadrV1R1})}\}} \ge 4 \mbox{ for all }\beta\in [0,1]
%\end{equation}
%(the expansion tends to infinity when $\beta$  tends to zero or to one)
%whereas the maximal amount of contraction when going from
%$\Q_{V_1,R_1}^\epsilon$ to  $\Q_{V_2,R_1}^\epsilon$
%is at most 
%\begin{equation}
%\dfrac
%{\min\{ \mbox{term in (\ref{eq:quadrV2R1}) }\}}
%{\max\{ \mbox{term in (\ref{eq:quadrV1R1}) }\}} \le 0.6  \mbox{ for all }\beta\in [0,1] .
%\label{EXCONV1V2}
%\end{equation}

\subsection{The dynamics of a Jitter map}
\label{ss:modelmap}

The dynamics near the periodic orbit $\Gamma$
is very complicated. As we have seen,
the Poincar\'e transition map to a section at a point of $\Gamma$ is a composition of
maps of the following form: 
$$x\mapsto A_1 \circ R_{2\pi/||y||+B(y)} \circ A_0^{-1}(x)\mbox{ where }y=A_0^{-1}(x).$$
%As we showed in Proposition~\ref{spiralest}
%orbits spiral as fast around the periodic orbit $\tilde \Gamma$ as they move parallel to it. 
In this section we shall first study the iterations of one of these maps.
In the next section we will then consider the composition of two 
suitable Jitter maps.

Let $\dist$ be a metric on $\R^2$ with the property
that each half-line through the origin intersects $\partial D_r$ in a unique point
(where $D_r:=\{z\in \rz^2;\, |z|:=\dist(z,0)=r\}$).
Next take
$R_t\colon \rz^2\to \rz^2$ to be the quadrilateral rotation, i.e. the unique map
so that  $\rz^2\times \rz \ni (z,t)\mapsto R_t(z)\in \rz^2$
is continuous, $R_0=id$, so that  for each $z\in \rz^2$,
$\dist(R_t(z),0)=\dist(z,0)$ and so that for each $x\in \rz^2\setminus \{0\}$ the
angles of $R_t(z)$ and $z$ differ by $t$.
If $\dist$ is the Euclidean metric then this coincides with the usual rotation,
but in our setting it is convenient to take for
$\dist$ the sum-metric on $\rz^2$ (defined by $\dist(z,z'):=||z-z'||:=|x-x'|+|y-y'|$
for $z=(x,y)$ and $z'=(x',y')$) in which case $t\mapsto R_t(z)$ moves each point
along a square $\partial D_r(0)$.

Next define two homeomorphisms $A_0,A_1\colon \rz^2\to \rz^2$ which
preserve the axes and which map quadrilaterals containing $0$
with corners on the axes to quadrilaterals of the form $\partial D_r=\{(x,y)\in \rz^2 \,;  \,\, |x|+|y|=r\}$.
Assume that there exist $\lambda<1<\mu$ so that
for each $t\in (\lambda,\mu)$ there exists a  smooth curve $l_t$ through $0$
so that  
\begin{equation}
\dist(A_0^{-1} \circ A_1 (z),0)=t \cdot \dist(z,0)\mbox{ for each }z\in l_t
\label{eq:excontr}
\end{equation}
and such that $A_0^{-1}A_1(l_t)$ is transversal to  $\partial D_r$ for each $r>0$.
Consider 
$$F(z)=A _1\circ R_{\theta(w)} \circ A_0^{-1}(z)$$
where 
$$\theta(w)=2\pi/ ||w||+B(w)\mbox{ , }w=A_0^{-1}(z) \mbox{ and }||w||=\dist(w,0)$$
and $w\mapsto B(w) \in \rz$ is a continuous function which converges to zero as $w\to 0$. 
We will refer to $F$ as a `jitter map'.
\medskip

Let us prove that such a Jitter map maps are 'chaotic'.

\begin{prop}[A jitter map has many periodic orbits]
\label{prop:modelmap}
Let $F$ be as above and assume (\ref{eq:excontr}). Then 
the map $F$ has
periodic orbits of arbitrary period in each neighbourhood
of $0$.
\end{prop}

\begin{prop}[A jitter map contains a shift with infinitely many symbols]
\label{prop:modelmap2}
Let $F$ be as above and assume (\ref{eq:excontr}). Then there exists $N_0$ so that 
for each sequence $k_i\in \nz$ satisfying
$$\lambda \le \dfrac{k_{i+1}}{k_i}\le \mu \mbox{ with }k_i\ge N_0 \label{eq:allowed}$$
there exist a sequence $\delta_i\in (0,1)$ and 
 $z\ne 0$ with
$||F^i(z)||\in (\frac{1}{k_i+1+\delta_i},\frac{1}{k_i+\delta_i})$ for all $i\ge 0$.
\end{prop}

\begin{prop}
\label{prop:modelmapsens}
Let $F$ be as above and assume (\ref{eq:excontr}). Then there exists $N_0$ so that 
for each sequence $k_i\in \nz$ there exists a sequence $\delta_i\in (0,1)$ so that
$$\lambda \le \dfrac{k_{i+1}}{k_i}\le \mu \mbox{ and }k_i\ge N_0$$
there exists $z\ne 0$ with
$||F^i(z)||\in (\frac{1}{k_i+1+\delta_i},\frac{1}{k_i+\delta_i})$ for all $i\ge 0$.
\end{prop}

\bigskip

The first proposition implies, for example, that
there is a sequence of fixed points of $F$
converging to $0$. 

The second proposition implies that $F$ contains a shift on infinitely many symbols.
Indeed, define the  annuli
$$Ann_k:=\{z\in \rz^2; \dfrac{1}{(k+2)}\le \vvert z \vvert \le  \dfrac{1}{k} \}.$$
The annuli with $k$ even are all disjoint.  Hence, taking $k_i\in 2\zz$  in Proposition~\ref{prop:modelmap2},
we get the existence of $x\in Ann_{k_0}$ with $F^i(x)\in Ann_{k_i}$ for all $i$.
If we would consider $k_i\in \{k_0,k_0+2\}$ this would give a one-sided shift on  two symbols,
but the proposition guarantees the existence of orbits which jump several annuli further in or out
(the number is determined by $\lambda$ and $\mu$). It follows that $F$ has positive topological entropy.
In fact, the topological entropy is infinite because for each $n$, it contains a full one-sided shift
of $n$ symbols (which has entropy $\log n$).  

\begin{cor}[A jitter map has sensitive dependence on initial conditions]
\label{cor:modelmapsens}
Let $F$ be as above and assume (\ref{eq:excontr}). Then $F$ has sensitive
dependence on initial conditions for all points in the set corresponding
to the shift on infinitely many symbols.
\end{cor}

\begin{figure}[htp]
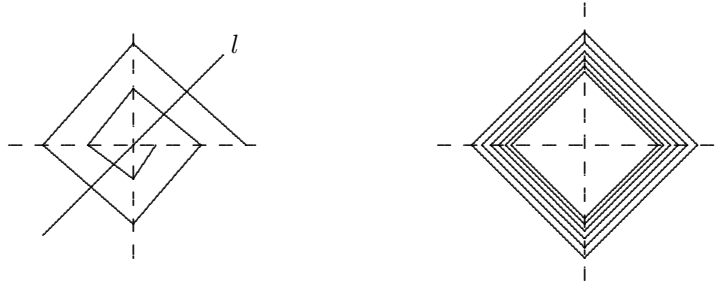
 \hfil
\beginpicture
\dimen0=0.15cm
\setcoordinatesystem units <\dimen0,\dimen0>
\setplotarea x from -20 to 20, y from -10 to 10
\setdashes
\plot -11 0 11 0 /
\plot 0 -10 0 10 /
\setsolid
\plot 10 0 0 9 -8 0 0 -7 6 0 0 5 -4 0 0 -3 2 0 /
\plot -8 -8 8 8 /
\put {$l$} at 9 9 
\setcoordinatesystem units <\dimen0,\dimen0> point at -40 0 
\setplotarea x from -20 to 20, y from -10 to 10
\setdashes
\plot -13 0 13 0 /
\plot 0 -12 0 13 /
\setsolid
\plot 10 0 0 10 -10 0 0 -10 10 0 /
\plot 9.1 0 0 9.1 -9.1 0 0 -9.1 9.1 0 /
\plot 8.3 0 0 8.3 -8.3 0 0 -8.3 8.3 0 /
\plot 7.6 0 0 7.6 -7.6 0 0 -7.6 7.6 0 /
\plot 7 0 0 7 -7 0 0 -7 7 0 /
\plot 6.5 0 0 6.5 -6.5 0 0 -6.5 6.5 0 /
%\plot 4.5 0 0 4.5 -4.5 0 0 -4.5 4.5 0 /
\endpicture
\caption{\label{fig:fixedpoints}{\small On the left, the image of a curve $l$ through $0$ is a spiral.
A conveniently chosen curve $l$ contains a sequence of fixed points, converging to $0$.
On the right, the sequence of annuli discussed below the statement of Proposition~\ref{prop:modelmap2}.
Orbits can jump between annuli according to these allowed sequences (\ref{eq:allowed}).}}
\end{figure}

\begin{proofof}{Proposition~\ref{prop:modelmap}}
Let us start by showing why $F$ has a sequence of
fixed points tending to $0$. 
It will be convenient to consider
$$\hat F(y):=A \circ R_{\theta(y)}(y)\mbox{ where }A=A_0^{-1}\circ A_1.$$  
By assumption there exists a curve $l$ through $0$ 
so that $\vvert A (z) \vvert =\vvert z \vvert $
for $z\in l$. Let $l_+$ be one of the two components of $l\setminus \{0\}$,
let $l'_+:=A(l_+)$
and let $m_+:=A_1(l_+)$.
We shall find a sequence of fixed points $y$ of $F$ on $m_+\setminus \{0\}$.
Indeed, for each $r>0$ small, 
let $\alpha(r)$ be the angle between the vectors $l_r'O$ and $l_rO$ 
where 
$l_r'\in l'_+$ and $l_r\in l_+$ are the unique points so that $\vert\vert   l_r'  \vert \vert =\vert\vert  l_r \vert \vert =r$.
Then choose $y\in l'_+$ so that 
\begin{equation}
\theta(y)=2\pi /\vert\vert  y\vert \vert +B(y)=\alpha(\vert \vert y \vert \vert) \mod 2\pi.
\label{eq:per1}
\end{equation}
Since $B(y)$ is bounded and continuous,
there exists a sequence of such points on $l'_+$ converging to $0$.
More precisely, 
there exists $\alpha$ so that for each sufficiently large $k\in \nz$
there exists $r\in (\frac{1}{k+1+\theta},\frac{1}{k+\theta})$
so that $y\in l'_+$ with $|y|=r$ satisfied (\ref{eq:per1}). 
So assume that (\ref{eq:per1}) holds.
Then for any such $y\in l'_+$ the point $x:=A_0y\in m_+$ is a fixed point of $F$.
Indeed, 
$R_{\theta(y)} (y)=R_{\alpha(|y|)}(y)\in l$ by the choice of $y$ and $\alpha$.
So 
\begin{equation}
\vvert \hat F(y) \vvert =
\vert \vert A \circ R_{\theta(y)} (y) \vert \vert =\vert \vert R_{\theta(y)} (y) \vert \vert =\vert \vert y \vert \vert 
\label{eq:l'1}
\end{equation}
and
\begin{equation}
A \circ R_{\theta(y)} (y)\in l'.
\label{eq:l'2}
\end{equation}
Since $y\in l'_+$ and $l'_+$ is a smooth curve which is transversal
to the quadrangles $D_r$,  equations
(\ref{eq:l'1}) and \ref{eq:l'2}) implies that 
$\hat F(y)=A \circ R_{\theta(y)} (y)=y$.
Hence
$A_0 \circ \hat F (y)=A_0(y)$
and so  $F(x)=A_0 \circ \hat F \circ A_0^{-1}(x)=x$ for $x=A_0(y)\in m_+$.

To explain the general case of periodic points of higher periods, let us show how to
construct a periodic orbit of period three (the general case
goes similarly).  
%For simplicity write $\hat F(y):=A_0^{-1}\circ A_1\circ R_{\phi(y)}(y)$
%and write $A=A_0^{-1}\circ A_1$.  
Consider the surface
$$
\tilde P(3)=\{(a_1,a_2,a_3)\st a_i\in \rz^+ \mbox{ and }a_1a_2a_3=1\}.$$
Choose $a_1,a_2,a_3\in ( \lambda , \mu)$
on this surface $\tilde P(3)$
and a disc neighbourhood $U$ of this point in this surface.
Associate to $a=(a_1,a_2,a_3)$ the curves
$l_i$, $i=1,2,3$ through $0$ such that $\vert \vert A(z) \vert \vert =a_i \vert \vert z \vert \vert$
for all $z\in l_i$ and let $l_{i,+}$ be a component of
$l_i\setminus \{0\}$. Let $l'_{i,+}=A(l_{i,+})$. For each $r>0$, let $\alpha_i(r)$ be
the angle between the $Ol'_{i,r}$ and $Ol_{i+1,r}$ where 
$l'_{i,r}=l'_{i,+}\cap \partial D_r$ and $l_{i+1,r}=l_{i+1,+}\cap \partial D_r$ 
(and where we take $l_{3+1}=l_1$).
Assume there exists $r>0$ and $a\in \tilde P(3)$ so that 
\begin{equation}
\begin{array}{llll}
\theta(l'_{1,r}) &= &\alpha_{1}(r)  &\mod 2\pi \\
\theta(l'_{2,a_2r}) &= & \alpha_{2}(a_2 r)  &\mod 2\pi\\
\theta(l'_{3,a_2a_3r})&= & \alpha_{3}(a_2a_3 r) & \mod2\pi.
\end{array}
\label{thrc}
\end{equation}
We claim that this implies that $x:=A_0(y)$ is a periodic point of  $F$ of period three
where $y=l_{1,r}$. Indeed, because $y=l_{1,r}= l'_{1,+}\cap D_r$ 
the first equation from (\ref{thrc}) implies that $R_{\theta(y)}(y) \in l_{2,+}$
and therefore that $\vvert \hat F(y) \vvert =
\vvert  A \circ R_{\theta(y)}(y) \vvert =a_2 \vvert y \vvert =a_2 r$
and $\hat F(y)=A\circ R_{\theta(y)}(y)\in l_{2,+}'$.
This and the second equation from (\ref{thrc}) implies that 
$R_{\theta(\hat F(x))}(\hat F(y)) \in l_{3,+}$
and so $\vvert \hat F \circ \hat F(y) \vvert = \vvert A \circ R_{\theta(\hat F(x))}(\hat F(y))\vvert =
a_3 \vvert |\hat F(y) \vvert  = a_2a_3 r$. Finally, this and
the third equation from (\ref{thrc})
implies that $R_{\hat F^2(y)}(\hat F^2(y))\in l_{1,+}$ and so 
$\vert\vert  \hat F^3(y) \vert\vert = a_1a_2a_3\vert \vert y \vert \vert $ and $\hat F^3(y)\in l_{1,+}'$.
Because $a_1 a_2 a_3=1$ and both $\hat F^3(y)$ and $y$ are in $l_{1,+}'$ we get therefore that $\hat F^3(y)=y$.  
Hence $F^3(x)=A_0 \circ \hat F^3 \circ A_0^{-1}(x)=x$.

So we need to show that (\ref{thrc}) has solutions.  Define
$G\colon \rz^+ \times \tilde P(3)
\to \rz \times \rz \times \rz$ by
$$G(r,a_1,a_2,a_3)=(\alpha_{1}(r),\alpha_{2}(a_2 r),\alpha_{3}(a_2a_3 r))$$% (\!\!\!\! \mod 2\pi)^3$$
where $y\in l'_{1,+}\cap D_r$ and $\alpha_i$ are the angles defined above.
This map is continuous and bounded. 
In fact, for each $\epsilon>0$ there exist $\delta>0$ and a neighbourhood $U$ in $\tilde P(3)$ 
as above
so that $G((0,\delta)\times U)$ is contained in  $\epsilon$-ball in $\rz^3$.
Next define the map
$$H\colon \rz^+\times \tilde P(3)\ni (r,a)
\mapsto (\theta(l'_{1,r}) , \theta(l'_{2,a_2r}) , \theta(l'_{3,a_2a_3r}) )
%\frac{1}{a_2a_3x})
\in \rz\times \rz\times \rz .$$
$H$ can be written as $H_1+H_2$ where 
$$H_1(r,a)=(2\pi/r,2\pi/(a_2 r),2\pi/(a_2a_3 r))\,\, \mbox { and } \,\, 
H_2(r,a)=(B(l'_{1,r}),B(l'_{2,a_2r}),B(l'_{3,a_2a_3 r})).$$
Note that (\ref{thrc}) is equivalent to
$G(r,a)=H(r,a)$, i.e. to  $G(r,a)-H_2(r,a)=H_1(r,a)$ (modulo $2\pi$).
The values of the left-side map
$(0,\delta) \times U\ni (r,a) \mapsto G(r,a)-H_2(r,a)$ are contained in a $2\epsilon$-ball, provided 
we choose $\delta>0$ so small that $|B(y)|\le \epsilon $ for all $y$ with $\vvert y \vvert \le  \delta$.
Note that  $H_1\colon (0,\delta) \times U \to \rz^3$ is invertible
and has inverse $H_1^{-1}(y_1,y_2,y_3)=(2\pi / y_1,y_1/y_2,y_2/y_3)$.
Hence $H_1^{-1}$ maps $2\pi \underline k + [0,2\pi]^3$, $\underline k=(k_0,k_1,k_2)\in \zz^3$
into $(0,\delta)\times U$ provided $|k_i|$ is large.
So $(G-H_2)\circ H_1^{-1}$ maps 
$2\pi \underline k + [0,2\pi]^3$ into some $2\epsilon$ ball. By Brouwer's fixed point theorem,
$(G-H_2)\circ H_1^{-1}$ has a fixed point (modulo $2\pi$) 
in  $2\pi \underline k  + \underline \kappa_0+ [0,2\pi]^3$ for some $\underline \kappa_0\in \rz^3$
(where $\underline k$ is arbitrary but with $|k_i|$ large. Hence 
there exists a constant $\kappa$ so that 
for each $k_0$ large,  $G(r,a)=H(r,a)$  (modulo $2\pi$) has a solution $(r,a)$
with $r\in (\frac{1}{k_0+1+\kappa},\frac{1}{k_0+\kappa})$ and $a\in U\subset \tilde P(3)$.
\end{proofof}

\begin{proofof}{Proposition~\ref{prop:modelmap2}}
The proof of the second assertion
also has a similar flavour, but to explain the proof more clearly we will 
assume that $B(z)=0$ and that the curve $l_t$ as in (\ref{eq:excontr}) are lines.
Again write $\hat F=A^{-1}\circ F$ where $A=A^{-1}_0\circ A_1$. 
Assume that $k_0,k_1,\dots$ is a sequence as in the assumption of the proposition,
and let $\kappa_0,\kappa_1,\dots\in (0,1)$ be a sequence to be determined later on.
%The idea is to consider an arbitrary integer $n$
%and first consider a piece of length $n$ of the orbit.
%So take $k_i$ as in the proposition, let $a_i>0$ 
%be so that $|a_i-k_{i+1}/k_i|< \epsilon_i a_i$.
Take $U_0=(1/(k_0+\kappa+1),1/k_0+\kappa_0)$
and inductively choose a sequence of intervals $U_1,U_2,\dots$ so that
$U_n$ is the set of  all $a_n\in \rz$ so that 
for all $x\in U_0$ and all $a_i\in U_i$, $i=1,\dots,n-1$,
$$1/(a_1\dots a_n x) \in (k_n+\kappa_n,k_n+\kappa_n+1).$$
Next let $l_0'=\rz^+$ and for $i\ge 1$, let $l_i$ be the half-line in the positive quadrant
with $\vert \vert  Ax \vert \vert =a_i\vert \vert x\vert \vert  $ on $l_i$.
Define $l'_i=Al_i$, $i\ge 1$ and for $i=0,1,2,\dots$ let $\alpha_i$ be the angle
between $l'_i$ and $l_{i+1}$. Note that $\alpha_i$ depends on $a_i$ and $a_{i+1}$. 
%Let
%$$\hat P(n)=\{(a_1,\dots,a_n)\st \vert \lambda \vert < a_i < \vert \mu \vert\}$$
%$$P'(n)=\{(a_2,\dots,a_n)\st (a_1,\dots,a_n)\in \hat P(n)\}.$$

We want to show that for each $n$ there exists a solution $x\in U_1$ and $a_i\in U_i$, $i=2,3,\dots$
of the system of equalities (analogous to (\ref{thrc})):
\begin{equation}
\begin{array}{rl}
2\pi/x &=\alpha_0(a_1) \mod 2\pi \\
2\pi /(a_2x) &= \alpha_1(a_1,a_2) \mod2\pi\\
& \vdots   \\
%2\pi/(a_{n-1} \cdot \dots \cdot a_3a_2x) &= \alpha_{n-2}(a_{n-2},a_{n-1}) \mod2\pi \\
2\pi/(a_n \cdot \dots \cdot a_3a_2x) &= \alpha_{n-1}(a_{n-1},a_{n}) \mod2\pi.
\end{array}
\label{thrcc}
\end{equation}
Let us show that this is enough.  Take $x\in \l_0'=\rz^+$ with 
distance  $x$ to $0$. Assume that we have  
$\hat F^i(x)\in l_i'$ and $\vert \vert \hat F^i (x)\vert\vert  =a_1\dots a_i x$.
Then the above equations give $R_{\theta(\hat F^i(x))}(\hat F^i(x))\in l_{i+1}$,
$\vert \vert \hat F^{i+1} (x)\vert\vert  = a_{i+1} \vvert \hat F^i(x) \vvert =
\vert  a_1\cdots a_{i+1} x\vert$  for $i=0,1,\dots,n$ and $\hat F^{i+1}(x)\in l_{i+1}'$.
Since $x\in U_0$ and $a_i\in U_i$, $i=1,2,\dots$ this proves by induction that
$$\vert \vert  \hat F^{i}(x) \vert \vert  
\in (\dfrac{1}{k_i+\delta_i+1},\dfrac{1}{k_i+\delta_i}).$$
for each $i=0,1,2,\dots$.

To prove that for each $n$ there exist
$x\in U_0$ and $a_i\in U_i$  as in (\ref{thrcc}), let
$$H(x,a_2,\dots,a_n)=\left( 2\pi /x,2\pi /(a_2x),
\dots,2\pi /(a_n a_{n-1} \cdots a_2 x)\right)$$
and   $$G(x,a_2,\dots,a_n)=(\alpha_{0}(a_1),\alpha_{1}(a_1,a_2),\alpha_{2}(a_2,a_3),\dots,\alpha_{n-1}(a_{n-1},a_n)).$$
$H$ is invertible with
$$H^{-1}(y_0,\dots,y_{n-1})=(2\pi/y_0,y_1/y_0,\dots,y_{n-1}/y_{n-2}).$$
Now take $\underline k=(k_0,k_1,\dots,k_{n-1})$ 
with  $k_i\ge N_0$ and with $\lambda \le k_{i+1}/k_i\le \mu$.
For each $\epsilon>0$ there exists $N_0$ so that
$H^{-1}$ maps $2\pi \underline k+[-2\pi,4\pi]^n$ into some $\epsilon$-neighbourhood
of $(k_0,k_1/k_0,\dots,k_{n-1}/k_{n-2})$. Then $G$ maps this neighbourhood into a neighbourhood
of some point in $\rz^n$. It follows that there exists $\underline \kappa=(\kappa_0,\dots,\kappa_{n-1})\in \rz^n$
so that $G\circ H^{-1}$ maps  $2\pi \underline k+2\pi \underline \kappa + [0,2\pi]^n$ into itself, and 
therefore has a fixed point $(y_0,\dots,y_{n-1})\in 2\pi \underline k+2\pi \underline \kappa + [0,2\pi]^n$. 
It follows $G=H$ (modulo $2\pi$) has as a solution $(x,a_2,\dots,a_n)=H(y_0,\dots,y_{n-1})$ of the required form. 
\end{proofof}

\subsection{The dynamics of the composition of two jitter maps and the Proof of Theorem~\ref{thm:inforbs}}
To prove Theorem~\ref{thm:inforbs}
note that we have seen in Proposition~\ref{spiralest}
that  the first entry maps $R_0$ and $R_1$
are of the form
$$x\mapsto A_1 \circ R_{2\pi/||y||+B(y)} \circ A_0^{-1}(x)\mbox{ where }y=A_0^{-1}(x)$$
and 
$$x\mapsto  A_3 \circ R_{2\pi/||y||+B'(y)} \circ A_2^{-1}(x)\mbox{ where }y=A_2^{-1}(x).$$
The first return map to the section associated to 
$V_0$ is the third iterate of $R_1\circ R_0$ (provided we identify 
the target space of $R_1$ appropriately with the domain space of $R_0$,
as we have done in the previous subsection, see for example
Figure~\ref{fig:axis}).
To show that this map has the required properties, we proceed 
as in the proof of Proposition~\ref{prop:modelmap} and write
$$\hat F(x)= A_0^{-1} \circ A_3 \circ R_{1/||y'||+B'(y')} \circ A_2^{-1} \circ A_1 \circ R_{1/||y||+B(y)}.$$
Note that $A_i$ is a piecewise linear map (linear on each quadrant),
and so we can describe these by four parameters (which determine the position
of each corner of the quadrilaterals). 
To compute the condition analogous to (\ref{eq:excontr}), for $A_2^{-1}\circ A_1$
we take the ratio of the $i$-th term in
(\ref{eq:quadrV1R0}) to the $i$-th term in 
(\ref{eq:quadrV1R1}):
$$\frac{(2-\beta)}{(\beta+1)}, \dfrac{(2-\beta)}{(1-\beta^2)}, \frac{(2-\beta)}{(\beta+1)\beta(1-\beta)}, \frac{(2-\beta)\beta}{(1-\beta^2)}.$$
The largest one of these (the third one) is $\ge 3.5$ for all $\beta\in (0,1)$ 
whereas the last one is $\le 1$ for all $\beta\in (0,1/2)$ (it is increasing)
and the first one is $\le 1$ for all $\beta(1/2,1)$ (it is decreasing).
So $|A_2^{-1}A_1(z)|/|z|$ can vary between $1$ and $3.5$.

To compute the condition analogous to (\ref{eq:excontr}), for $A_3^{-1}\circ A_0$
we take the ratio of the $i$-th term in
(\ref{eq:quadrV0R0}) to the $i$-th term in 
(\ref{eq:quadrV2R1}):
$$\dfrac{(2-\beta-\beta^2)}{(1+2\beta)}, \dfrac{(2-\beta-\beta^2)(1+\beta)}{(1+2\beta)\beta}, 
\dfrac{\beta(11+21\beta+15\beta^2+7\beta^3)}{(2+\beta)^2(1+2\beta)}, \frac{(2-\beta-\beta^2)}{(1+2\beta)(1+\beta)}.$$
The largest of these is either the 2nd or the 3rd, and the maximum of these two is $\ge 1.2$ for all $\beta\in (0,1)$.
The third one is increasing and the last one increasing, with the
first one $<0.8$ for $\beta\in (0,0.35)$ and the last one $<0.8$ for $\beta\in (0.35,1)$.
So $|A_0^{-1}A_3(z)|/|z|$ can vary between $0.8$ and $3.5$.

So the condition corresponding to (\ref{eq:excontr}) holds. 
It follows that $\hat F$ has a sequence of fixed points (and periodic orbits) as before,
and that the properties as in Theorem~\ref{thm:inforbs} hold.

\section{Proof of Theorem~\ref{thm:robustness}.}
\label{sec:robustness}

Take 
matrices $A$ and $B$ with  $$||A-A_\beta ||,||B-B_\beta || \le \epsilon$$
where $|| \,\,\, ||$ stands for some matrix norm.

That the Shapley and anti-Shapley orbit exists for $A,B$ near $A_\beta,B_\beta$, simply follows
from the hyperbolicity of the first return map to a section transversal to these orbits.
(The first return maps are projective transformations.)

So let us discuss the persistence of the orbit $\Gamma$.
Provided $\epsilon>0$ is small enough, the set $\Sigma_A$ and $\Sigma_B$
are still divided up in three regions meeting in a Nash equilibrium as in Figure~\ref{fig:Jorbit}.
(The angles between the lines will no longer be necessarily equal and the Nash equilibria will no longer
be in the barycentre.) Now again consider the induced flow on the boundary.
The set where one or two players are indifferent are still arranged as in Figures~\ref{figs3}
and \ref{figs3gamma}  (they change continuously with $A$ and $B$). 
The part of $\partial \Sigma$ where both players are indifferent still consists of a closed curve
$\Gamma$ along which orbits spiral (along cones as in 
Figure~\ref{fig:quad}), and the other part through which orbits cross transversally.
The transition maps can be computed along this orbit as was done in 
Section~\ref{subsec:analyticcomp}, but in any case, the quadrangles computed
in that section depend again continuously on $A$ and $B$. So it follows that for fictitious
play associated to $A$ and $B$
sufficiently close to $A_\beta,B_\beta$, one still has the existence of a sequence
of periodic orbits for the flow induced on the boundary.

Next we argue for the original system. The cone over the hexagonal $\Gamma$ with apex
the Nash equilibrium $E$ is completely invariant and depends continuously on $A,B$.
Moreover, this cone is two dimensional (but of course embedded in the four-dimensional space $\Sigma$). 
So now take a half-line $l$ in this cone through the apex, and consider the first return map to $l$.
Because of the general form of the return maps, this first return map $R\colon l\to l$
is a Moebius transformation, with a fixed point at $E$. As we showed
in the appendix of Sparrow et al \cite{SSH2008}, taking $A_\beta,B_\beta$ we have the following:
for $\beta\in (0,\sigma)$ the fixed point $E$ of 
$R\colon l\to l$ is attracting, for $\beta=\sigma$ it is neutral, and for $\beta\in (\sigma,1)$
repelling and another fixed point appears which attracts all points in $l\setminus \{E\}$,
because the map is a Moebius map. (For $\beta\in (0,\sigma)$, the map $R\colon l\to l$
also has a 'virtual' 2nd fixed point,  corresponding to the 'negative' part of the half-line $l$.)
For $A,B$ close to $A_\beta,B_\beta$ the corresponding first return maps are also near
those of $A_\beta,B_\beta$. So when $\beta\ne \sigma$, and 
$A,B$ is sufficiently close to $A_\beta,B_\beta$ we have the same behaviour
for $\Gamma$.

Using Proposition~\ref{prop:implicationsigma} one can get the same conclusions
for the other periodic orbits for the original flow.

\section{Conclusion}
\label{s:concl}

For $\beta\in (-1,0]$ players always asymptotically 
become periodic.  When $\beta\in (0,\sigma)$
the Shapley orbit is still attracting but not globally attracting: 
there is an abundance of orbits (many with periodic play) as described in 
Theorem~\ref{thm:main} which tend to $E$. 
For $\beta\in (\sigma,1)$ we have chaos
(in the sense that there exist  subshifts of finite type).
This chaos is caused by a periodic orbit $\Gamma$ whose
first return map is of what we call `jitter type': orbits can move
further away and closer to the periodic orbit $\Gamma$ in a manner
which is reminiscent to that of a random walk.
Numerical simulations suggest that for $\beta\in (\sigma,\tau)$
this is precisely what happens for most starting points.
However, when $\beta\in (\tau,1)$ there exists an attracting
anti-Shapley orbit, but again this orbit is not globally attracting.

As shown in Theorem~\ref{thm:robustness} the analysis we give
does not depend on the symmetry of our matrices. The symmetric
matrices $A_\beta,B_\beta$ merely simplified our calculations, but 
using simple perturbation arguments the results also apply
to nearby maps.

%\bigskip

%Of course most of the analysis we described also could be
%applied to other families (with less symmetry). 

Given the relationship established in Gaunersdorfer \& Hofbauer \cite{GH95}
between fictitious play and replicator dynamics, it would be interesting
to see whether chaos can also occur in the latter.\footnote{As we were about to finish
writing this paper, we received a preprint by Aguiar and Castro \cite{Aguiar08}, showing
that in two-person with replicator dynamics one can have chaotic switching
between strategies. There, the mechanism causing chaos is slightly different:
the dynamical system is smooth and chaos is caused by homoclinic cycles associated
to the 9 singularities of the system.  More precisely, some of the singularities are part
of several  distinct cycles, from which the authors are able to deduce that
there are orbits which switch between following one cycle and another.}

\section{Acknowledgements}

We would like to thank Chris Harris, who introduced us to this problem.
This paper grew out a previous joint paper with him, see Sparrow et al \cite{SSH2008}.
%referee and an associate editor for helpful comments and suggestions.

\end{document}